\documentclass[11pt]{amsart}

\usepackage{amsmath,amsfonts,amssymb,shuffle,color}
\usepackage{enumitem}

\newtheorem{theorem}{Theorem}[section]
\newtheorem{lemma}[theorem]{Lemma}
\newtheorem{proposition}[theorem]{Proposition}
\newtheorem{corollary}[theorem]{Corollary}

\theoremstyle{definition}
\newtheorem{definition}[theorem]{Definition}

\theoremstyle{remark}
\newtheorem{remark}[theorem]{Remark}

\numberwithin{equation}{section}

\newcommand{\inv}{\ensuremath\mathrm{inv}}
\newcommand{\Des}{\ensuremath\mathrm{Des}}
\newcommand{\des}{\ensuremath\mathrm{des}}
\newcommand{\wt}{\ensuremath\mathrm{wt}}

\newcommand{\R}{\ensuremath{R}}

\newcommand{\YD}{\ensuremath\mathbb{D}}
\newcommand{\descend}{\ensuremath\mathbb{D}}
\newcommand{\ascend}{\ensuremath\mathbb{A}}

\newcommand{\SYT}{\ensuremath\mathrm{SYT}}
\newcommand{\Key}{\ensuremath\mathrm{SKT}}

\newcommand{\QKT}{\ensuremath\mathrm{QKT}}

\newcommand{\swap}{\ensuremath\mathfrak{s}}
\newcommand{\braid}{\ensuremath\mathfrak{b}}

\newcommand{\flatten}{\ensuremath\mathrm{flat}}
\newcommand{\sort}{\ensuremath\mathrm{sort}}

\newcommand{\stanley}{\ensuremath\mathrm{S}}
\newcommand{\schubert}{\ensuremath\mathfrak{S}}
\newcommand{\key}{\ensuremath\kappa}
\newcommand{\fund}{\ensuremath\mathfrak{F}}

\newlength\cellsize 

\savebox2{%
\begin{picture}(12,12)
\put(0,0){\line(1,0){12}}
\put(0,0){\line(0,1){12}}
\put(12,0){\line(0,1){12}}
\put(0,12){\line(1,0){12}}
\end{picture}}

\newcommand\cellify[1]{\def\thearg{#1}\def\nothing{}%
\ifx\thearg\nothing\vrule width0pt height\cellsize depth0pt%
  \else\hbox to 0pt{\usebox2\hss}\fi%
  \vbox to 12\unitlength{\vss\hbox to 12\unitlength{\hss$#1$\hss}\vss}}

\newcommand\tableau[1]{\vtop{\let\\=\cr
\setlength\baselineskip{-12000pt}
\setlength\lineskiplimit{12000pt}
\setlength\lineskip{0pt}
\halign{&\cellify{##}\cr#1\crcr}}}

\newcommand{\graybox}{\textcolor[RGB]{220,220,220}{\rule{1\cellsize}{1\cellsize}}\hspace{-\cellsize}\usebox2}

\newcommand\gridify[1]{\vbox to 10\unitlength{\vss\hbox to 10\unitlength{\hss$_{#1}$\hss}\vss}}

\newcommand\wires[1]{\vtop{\let\\=\cr
\setlength\baselineskip{-10000pt}
\setlength\lineskiplimit{10000pt}
\setlength\lineskip{0pt}
\halign{&\gridify{##}\cr#1\crcr}}}

\begin{document}


\title[Weak dual equivalence]{Weak dual equivalence for polynomials}  

\author[S. Assaf]{Sami Assaf}
\address{Department of Mathematics, University of Southern California, Los Angeles, CA 90089}
\email{shassaf@usc.edu}





\begin{abstract}
  We use dual equivalence to give a short, combinatorial proof that Stanley symmetric functions are Schur positive. We introduce weak dual equivalence, and use it to give a short, combinatorial proof that Schubert polynomials are key positive. To demonstrate further the utility of this new tool, we use weak dual equivalence to prove a nonnegative Littlewood--Richardson rule for the key expansion of the product of a key polynomial and a Schur polynomial, and to introduce skew key polynomials that, when skewed by a partition, expand nonnegatively in the key basis. 
\end{abstract}

\maketitle
\tableofcontents

%
\section{Introduction}
%
\label{sec:introduction}

Schur functions enjoy deep connections with representation theory and algebraic geometry. The quintessential problem of proving that a given function expands nonnegatively in the Schur basis arises because these Schur coefficients enumerate multiplicities of irreducible components or dimensions of algebraic varieties. In earlier work, the author developed a general framework, called \emph{dual equivalence}, for proving that a given function is symmetric and Schur positive\cite{Ass15}. At its core, the method imposes a rigid structure on the set of combinatorial objects that generate the given function that ensures there is a weight-preserving bijection with standard Young tableaux, the latter of which generate Schur functions.

In this paper, we begin with a new application of dual equivalence to establish that \emph{Stanley symmetric functions} \cite{Sta84}, introduced by Stanley as a tool to enumerate reduced expressions for permutations, are Schur positive. While Edelman and Greene prove this using a complicated insertion algorithm \cite{EG87}, the new proof we present is a vast simplification. Furthermore, this application sets the stage for a generalization of dual equivalence to the general (not necessarily symmetric) polynomial setting with ramifications in representation theory and geometry.

Macdonald noted that Stanley symmetric functions are stable limits of \emph{Schubert polynomials} \cite{Mac91}, introduced by Lascoux and Sch{\"u}tzenberger \cite{LS82} as polynomial representatives of Schubert classes for the cohomology of the flag manifold. One of the fundamental open problems in algebraic combinatorics is to give a nonnegative rule for the Schubert expansion of a product of Schubert polynomials. Geometrically, these coefficients give intersection numbers for suitable intersections of Schubert varieties, but as yet there is no combinatorial proof that they are nonnegative.

Related to this, Demazure studied characters for certain general linear group modules \cite{Dem74} that coincide with \emph{key polynomials} studied combinatorially by Lascoux and Sch{\"u}tzenberger \cite{LS90}. These key polynomials enjoy rich geometric interpretations, but their structure constants are not, in general, nonnegative, though Haglund, Luoto, Mason and van Willigenburg proved that they are in the special case when one of the polynomials is symmetric \cite{HLMvW11}. Nevertheless, Lascoux  and Sch{\"u}tzenberger proved that Schubert polynomials expand nonnegatively in the key basis \cite{LS90}. Moreover, the stable limit of a key polynomial is a Schur polynomial, so this expansion may be regarded as a polynomial pull-back of the Schur expansion of a Stanley symmetric function.

Developing this idea, we generalize dual equivalence into \emph{weak dual equivalence} that provides a general framework for proving that a given polynomial is nonnegative in the key basis. We do this by developing a combinatorial model for key polynomials, which we call \emph{standard key tableaux}, and using this to define a similarly rigid structure on combinatorial objects that ensures a weight-preserving bijection with standard key tableaux, thereby ensuring key positivity. An immediate application is that the exact structure that gives Schur positivity of Stanley symmetric functions by dual equivalence also gives key positivity of Schubert polynomials by weak dual equivalence. 

Pushing this still further, we recall from \cite{Ass15} how dual equivalence gives a simple Littlewood--Richardson rule for the Schur expansion of a product of Schur functions, and present an analogous model for the product of key polynomials. In so doing, we see precisely why a general product of key polynomials is not nonnegative in the key basis, and recover a remarkably simplified proof of Haglund, Luoto, Mason and van Willigenburg's Littlewood--Richardson rule for the key expansion of the product of a key polynomial and Schur polynomial. Following the analogy with Schur functions, we also define skew key polynomials that, when skewed by a partition shape, are nonnegative sums of key polynomials. While not nonnegative in general, we prove that these skew key polynomials always stablize to Schur positive functions.

%

%
\section{Schur positivity of Stanley symmetric functions}
%
\label{sec:symmetric}

\subsection{Schur functions}
\label{sec:schur}

Let $\mathbb{N}$ and $\mathbb{P}$ denote the sets of nonnegative and positive integers, respectively. We use letters $a,b,c$ to denote weak compositions of length $n$, i.e. sequences in $\mathbb{N}^n$, letters $\alpha,\beta,\gamma$ to denote strong compositions, i.e. sequences in $\mathbb{P}^k$ for some $k$, and $\lambda,\mu,\nu$ to denote partitions, i.e. weakly decreasing sequences in $\mathbb{P}^k$ for some $k$. Given a weak composition $a$, let $\flatten(a)$ denote the strong composition obtained by removing all zero parts. Given a strong composition $\alpha$, let $\sort(\alpha)$ denote the partition obtained by rearranging the parts of $\alpha$ into weakly decreasing order, and extend this to weak compositions by $\sort(a) = \sort(\flatten(a))$.

Given a weak composition $a$, we let $x^a$ denote the monomial $x_1^{a_1} \cdots x_n^{a_n}$, and we use the notation $X$ to denote the infinite set of variables $\{x_1,x_2,\ldots\}$ used for functions.

For our discussion of symmetric functions, we defer to the beautiful exposition in Macdonald \cite{Mac95}, though we use coordinate notation (French) as opposed to matrix notation (English). The \emph{Young diagram} of a partition $\lambda$, denoted by $\YD(\lambda)$, is the diagram with $\lambda_i$ unit cells left justified in row $i$. For example, the Young diagram for $(5,4,4,1)$ is given in Figure~\ref{fig:Young}.

\begin{figure}[ht]
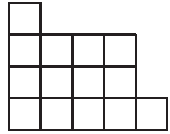

  \begin{displaymath}
    \tableau{ \ \\ \ & \ & \ & \ \\ \ & \ & \ & \ \\ \ & \ & \ & \ & \ \\}
  \end{displaymath}
  \caption{\label{fig:Young}The Young diagram for $(5,4,4,1)$.}
\end{figure}

The \emph{Schur functions}, indexed by partitions, form an important basis for symmetric functions. They arise in many contexts, including as irreducible characters for the general linear group and as polynomial representatives for Schubert cycles of the Grassmannian. For our purposes, we define them combinatorially as the \emph{quasisymmetric generating function} for standard Young tableaux.

Gessel introduced the \emph{fundamental quasisymmetric functions} \cite{Ges84}, indexed by strong compositions, that form an important basis for quasisymmetric functions. Given strong compositions $\alpha,\beta$, we say that \emph{$\beta$ refines $\alpha$} if there exist indices $i_1<\ldots<i_k$ such that
\begin{displaymath}
  \beta_1 + \cdots + \beta_{i_j} = \alpha_1 + \cdots + \alpha_j.
\end{displaymath}
For example, $(1,2,2)$ refines $(3,2)$ but does not refine $(2,3)$.

\begin{definition}[\cite{Ges84}]
  For $\alpha$ a strong composition, the fundamental quasisymmetric function $F_{\alpha}$ is given by
  \begin{equation}
    F_{\alpha}(X) = \sum_{\flatten(b) \ \mathrm{refines} \ \alpha} x^b,
    \label{e:F_n}
  \end{equation}
  where the sum is over weak compositions $b$ whose flattening refines $\alpha$.
\end{definition}
  
For example, restricting to three variables to make the expansion finite, we have
\begin{displaymath}
  F_{(3,2)}(x_1,x_2,x_3) = x^{032} + x^{302} + x^{320} + x^{311} + x^{122} + x^{212}.
\end{displaymath}

A \emph{standard Young tableau} is a bijective filling of a Young diagram with entries $\{1,2,\ldots,n\}$ such that entries increase along rows and up columns. Let $\SYT(\lambda)$ denote the set of standard Young tableaux of shape $\lambda$. For example, Figure~\ref{fig:SYT} shows the standard Young tableaux of shape $(3,2)$. 

For a standard Young tableau $T$, say that $i$ is a \emph{descent of $T$} if $i+1$ lies weakly left of (equivalently, strictly above) $i$. The \emph{descent composition of $T$}, denoted by $\Des(T)$, is the strong composition given by maximal length runs between descents. For example, see Figure~\ref{fig:SYT}.

\begin{figure}[ht]
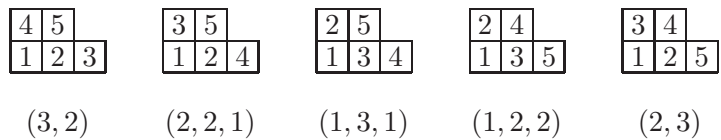

  \begin{displaymath}
    \begin{array}{c@{\hskip 2em}c@{\hskip 2em}c@{\hskip 2em}c@{\hskip 2em}c}
      \tableau{4 & 5 \\ 1 & 2 & 3} & 
      \tableau{3 & 5 \\ 1 & 2 & 4} & 
      \tableau{2 & 5 \\ 1 & 3 & 4} &
      \tableau{2 & 4 \\ 1 & 3 & 5} & 
      \tableau{3 & 4 \\ 1 & 2 & 5} \\ \\
      (3,2) &
      (2,2,1) &
      (1,3,1) &
      (1,2,2) &
      (2,3) 
    \end{array}
  \end{displaymath}
  \caption{\label{fig:SYT}The standard Young tableaux for $(3,2)$ and their descent compositions.}
\end{figure}

The following definition for a Schur function follows from the classical one (see \cite{Mac95}) by a result due to Gessel \cite{Ges84}.

\begin{definition}
  For $\lambda$ a partition, the Schur function $s_{\lambda}$ is given by
  \begin{equation}
    s_{\lambda}(X) = \sum_{T \in \SYT(\lambda)} F_{\Des(T)}(X).
    \label{e:schur-F}
  \end{equation}
  \label{def:schur-F}
\end{definition}

For example, from Figure~\ref{fig:SYT} we compute
\begin{displaymath}
  s_{(3,2)}(X) = F_{(2,3)}(X) + F_{(1,2,2)}(X) + F_{(1,3,1)}(X) + F_{(3,2)}(X) + F_{(2,2,1)}(X).
\end{displaymath}

The \emph{reverse row reading word} of a Young tableau $T$ is obtained by reading the enties right to left along the top row, then right to left for the next row down, and so on. We say that a standard Young tableau is \emph{super-standard} if its reverse row reading word is the reverse of the identity. For examples, see Figure~\ref{fig:super-SYT}.

\begin{figure}[ht]
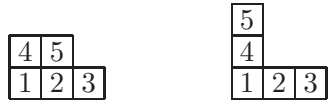

  \begin{displaymath}
    \tableau{\\ 4 & 5 \\ 1 & 2 & 3} \hspace{4\cellsize} \tableau{5 \\ 4 \\ 1 & 2 & 3}
  \end{displaymath}
  \caption{\label{fig:super-SYT}The super-standard Young tableaux for $(3,2)$ (left) and $(3,1,1)$ (right).}
\end{figure}

\begin{proposition}
  For any partition $\lambda$, there exists a unique super-standard Young tableau $Y$ of shape $\lambda$. Moreover, for $T\in\SYT(\lambda)$, we have $\Des(T) = \lambda$ if and only if $T=Y$.
  \label{prop:super-SYT}
\end{proposition}

\begin{proof}
  Clearly $Y$ is unique, if it exists, and it can be constructed by placing entries $1,\ldots,\lambda_1$ in the bottom row, $\lambda_1+1,\ldots,\lambda_1+\lambda_2$ in the next row up, and so on. Then $i$ will be a descent of $Y$ precisely for $i=\lambda_1,\lambda_1+\lambda_2,\ldots$, giving $\Des(Y)=\lambda$. Conversely, if $\Des(T)=\lambda$, then $T$ must have descents at $i$ precisely for $i=\lambda_1,\lambda_1+\lambda_2,\ldots$. Therefore the entries $1,2,\ldots,\lambda_1$ must form a horizontal strip, and so must fill the bottom row of $\lambda$. Similarly, $\lambda_1+1,\ldots,\lambda_1+\lambda_2$ must fill the next row up, and so on, giving $T=Y$.
\end{proof}

\subsection{Stanley symmetric functions}
\label{sec:stanley}

A \emph{reduced expression} is a sequence $\rho = (i_k, \ldots, i_1)$ such that the permutation $s_{i_k} \cdots s_{i_1}$ has $k$ inversions, where $s_i$ is the simple transposition that interchanges $i$ and $i+1$. Define the set $\R(w)$ of reduced expressions for $w$ by
\begin{equation}
  \R(w) = \{ (i_{\inv(w)}, \ldots, i_{1}) \mid w = s_{i_{\inv(w)}} \cdots s_{i_1} \}.
\end{equation}
For example, the elements of $\R(42153)$ are shown in Figure~\ref{fig:red}.

\begin{figure}[ht]
  \begin{displaymath}
    \begin{array}{cccccc}
      (4,2,1,2,3) & (4,1,2,1,3) & (4,1,2,3,1) & (2,4,1,2,3) & (2,1,4,2,3) & (2,1,2,4,3) \\
      (1,4,2,3,1) & (1,2,4,3,1) & (1,4,2,1,3) & (1,2,4,1,3) & (1,2,1,4,3) & 
    \end{array}
  \end{displaymath}
   \caption{\label{fig:red}The set of reduced expressions for $42153$.}
\end{figure}

\begin{definition}
  The \emph{run decomposition} of a reduced expression $\rho$, denoted by $\rho^{(k)} | \cdots | \rho^{(1)}$, partitions $\rho$ into increasing sequences of maximal length. The \emph{descent composition of $\rho$}, denoted by $\Des(\rho)$, is the strong composition $(|\rho^{(1)}|,\ldots,|\rho^{(k)}|)$. 
  \label{def:Des}
\end{definition}

For example, the run decomposition of $(1,4,2,3,1)$ is $(14 | 23 | 1)$ and so $\Des(1,4,2,1,3) = (1,2,2)$. Note the reversal of lengths. 

In order to enumerate reduced expressions, Stanley defined a family of symmetric functions indexed by permutations that are the generating functions for reduced expressions.

\begin{definition}[\cite{Sta84}]
  For $w$ a permutation, the \emph{Stanley symmetric function} $\stanley_{w}$ is
  \begin{equation}
    \stanley_{w}(X) = \sum_{\rho \in \R(w)} F_{\Des(\rho)}(X),
    \label{e:stanley-F}
  \end{equation}
  where the sum is over all reduced expressions for $w$.
  \label{def:stanley-F}
\end{definition}

To avoid confusion with fundamental quasisymmetric functions, we diverge from usual notation of $F_w$ and denote the Stanley symmetric functions by $\stanley_{w}$. Also note that we follow usual conventions and have our $\stanley_w = F_{w^{-1}}$ in \cite{Sta84}.

For example, from Figure~\ref{fig:red}, we compute
\begin{eqnarray*}
  \stanley_{42153}(X) & = & F_{(3,1,1)}(X) + 2 F_{(2,2,1)}(X) + 2 F_{(1,3,1)}(X) + F_{(3,2)}(X) \\
  & & + 2 F_{(1,2,2)}(X) + F_{(1,1,3)}(X) + F_{(2,1,2)}(X) + F_{(2,3)}(X) .
\end{eqnarray*}

Not only are the Stanley symmetric functions honest symmetric functions \cite{Sta84}, Edelman and Greene \cite{EG87} showed that they are, in fact, Schur positive. For example,
\begin{displaymath}
  \stanley_{42153}(X) = s_{(3,2)}(X) + s_{(3,1,1)}(X).
\end{displaymath}
We give an independent and elementary proof of this fact using dual equivalence.

\subsection{Dual equivalence}
\label{sec:graphs}

Given a set of combinatorial objects $\mathcal{A}$ endowed with a notion of descents, one can form the quasisymmetric generating function for $\mathcal{A}$ by
\begin{displaymath}
  \sum_{T \in \mathcal{A}} F_{\Des(T)}(X).
\end{displaymath}
Two examples of this are Schur functions generated by standard Young tableaux \eqref{e:schur-F} and Stanley symmetric functions generated by reduced expressions \eqref{e:stanley-F}.

Dual equivalence \cite{Ass15} is a general framework for proving that such generating functions are symmetric and Schur positive. We recall the relevant definitions and theorems, then apply them to $\R(w)$ to prove that Stanley symmetric functions are symmetric and Schur positive.

\begin{definition}[\cite{Ass15}]
  Let $\mathcal{A}$ be a finite set, and $\Des$ be a map from $\mathcal{A}$ to strong compositions of $n$. A \emph{dual equivalence for $(\mathcal{A},\Des)$} is a family of involutions $\{\varphi_i\}_{1<i<n}$ on $\mathcal{A}$ such that
  \renewcommand{\theenumi}{\roman{enumi}}
  \begin{enumerate}
  \item For all $i-h \leq 3$ and all $T \in \mathcal{A}$, there exists a partition $\lambda$ of $i-h+3$ such that
    \[ \sum_{U \in [T]_{(h,i)}} F_{\Des_{(h-1,i+1)}(U)}(X) = s_{\lambda}(X), \]
    where $[T]_{(h,i)}$ is the equivalence class generated by $\varphi_h,\ldots,\varphi_i$, and $\Des_{(h,i)}(T)$ is the strong composition of $i-h+1$ obtained by deleting the first $h-1$ and last $n-i$ parts from $\Des(T)$.
    
  \item For all $|i-j| \geq 3$ and all $T \in\mathcal{A}$, we have
    \begin{displaymath}
      \varphi_{j} \varphi_{i}(T) = \varphi_{i} \varphi_{j}(T).
    \end{displaymath}

  \end{enumerate}

  \label{def:deg}
\end{definition}

For example, if $\Des(T) = (3,2,3,1)$, then we have $\Des_{(3,8)}(T) = (1,2,3)$.

Haiman \cite{Hai92} defined involutions $d_i$ on standard Young tableaux that swap $i$ with $i \pm 1$ whenever $i \mp 1$ lies in between them in the column reading word (read top to bottom from left to right). Assaf \cite{Ass15} showed that these involutions satisfy Definition~\ref{def:deg} and that any involutions satisfying Definition~\ref{def:deg} have $\Des$-isomorphic equivalence classes. In particular, we have the following.

\begin{figure}[ht]
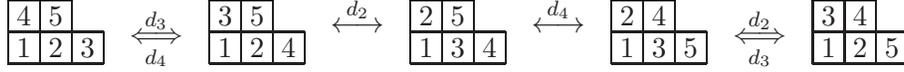

  \begin{displaymath}
    \begin{array}{ccccccccc}
      \tableau{4 & 5 \\ 1 & 2 & 3} & \raisebox{-\cellsize}{$\stackrel{\displaystyle\stackrel{d_3}{\Longleftrightarrow}}{\scriptstyle d_4}$}  &
      \tableau{3 & 5 \\ 1 & 2 & 4} & \raisebox{0\cellsize}{$\stackrel{d_2}{\longleftrightarrow}$} &                                                
      \tableau{2 & 5 \\ 1 & 3 & 4} & \raisebox{0\cellsize}{$\stackrel{d_4}{\longleftrightarrow}$} &                                                
      \tableau{2 & 4 \\ 1 & 3 & 5} & \raisebox{-\cellsize}{$\stackrel{\displaystyle\stackrel{d_2}{\Longleftrightarrow}}{\scriptstyle d_3}$} & 
      \tableau{3 & 4 \\ 1 & 2 & 5} 
    \end{array}
  \end{displaymath}
  \caption{\label{fig:SYT-deg}Dual equivalence for the standard Young tableaux for $(3,2)$.}
\end{figure}

\begin{theorem}[\cite{Ass15}]
  If $\{\varphi_i\}$ is a dual equivalence for $(\mathcal{A},\Des)$, and $U \in \mathcal{A}$, then
  \begin{equation}
    \sum_{T \in [U]} F_{\Des(T)}(X) \ = \ s_{\lambda}(X)
  \end{equation}
  for some partition $\lambda$. In particular, the fundamental quasisymmetric generating function for $\mathcal{A}$ is symmetric and Schur positive.
  \label{thm:positivity}
\end{theorem}

We use the defining relations for the simple transpositions that generate the symmetric group  to construct involutions on $\R(w)$ that satisfy Definition~\ref{def:deg}.

\begin{definition}
  Given $\rho \in\R(w)$ and $1 < i < \inv(w)$, define $\mathfrak{d}_i(\rho)$ by
  \begin{equation}
    \mathfrak{d}_i (\rho) = \left\{ \begin{array}{rl}
      \braid_i   (\rho) & \mbox{if } \rho_{i+1} = \rho_{i-1} = \rho_{i} \pm 1 \\
      \swap_{i-1}(\rho) & \mbox{if } \rho_{i-1} > \rho_{i+1} > \rho_{i} \mbox{ or } \rho_{i-1} < \rho_{i+1} < \rho_{i}, \\
      \swap_{i}  (\rho) & \mbox{if } \rho_{i+1} > \rho_{i-1} > \rho_{i}  \mbox{ or } \rho_{i+1} < \rho_{i-1} < \rho_{i}, \\
      \rho & \mbox{otherwise},
    \end{array} \right.
  \end{equation}
  where $\braid_j$ changes $\rho_{j-1} \rho_j \rho_{j-1}$ to $\rho_{j} \rho_{j-1} \rho_{j}$; and $\swap_j$ interchanges $\rho_j$ and $\rho_{j+1}$.
  \label{def:dual-stanley}
\end{definition}

For example, one of the two dual equivalences class of $\R(42153)$ is shown in Figure~\ref{fig:dual}. Observe that the generating function for this class is the Schur function $s_{(3,2)}(X)$, picking off one term in the expansion $\stanley_{42153}(X) = s_{(3,2)}(X) + s_{(3,1,1)}(X)$.

\begin{figure}[ht]
  \begin{displaymath}
    (1,2,4,1,3) \stackrel{\displaystyle\stackrel{\mathfrak{d}_3}{\Longleftrightarrow}}{\scriptstyle\mathfrak{d}_4} 
    (1,2,1,4,3) \stackrel{\mathfrak{d}_2}{\longleftrightarrow}                                                     
    (2,1,2,4,3) \stackrel{\mathfrak{d}_3}{\longleftrightarrow}                                                     
    (2,1,4,2,3) \stackrel{\displaystyle\stackrel{\mathfrak{d}_2}{\Longleftrightarrow}}{\scriptstyle\mathfrak{d}_3} 
    (2,4,1,2,3) 
  \end{displaymath}
  \caption{\label{fig:dual}Examples of dual equivalence on $\R(42153)$.}
\end{figure}

\begin{remark}
  Note that if $\rho$ contains distinct indices, then changing those indices, in order, to $1\ldots n$ and taking the inverse of the resulting permutation gives a bijection, say $\theta$, with permutations, and we have $\theta(\mathfrak{d}_i(\rho)) = d_i(\theta(\rho))$, where $d_i$ is Haiman's dual equivalence involutions on permutations. Therefore, in this case, it follows from \cite{Ass15} that the maps $\mathfrak{d}_i$ give a dual equivalence for $\R(w)$. 
\end{remark}
  
\begin{theorem}
  The maps $\{\mathfrak{d}_i\}$ give a dual equivalence for $\R(w)$. In particular, Stanley symmetric functions are symmetric and Schur positive.
  \label{thm:deg-red}
\end{theorem}

\begin{proof}
  Both swaps, $\swap_i$, and braids, $\braid_i$, are themselves involutions on $\R(w)$ provided they apply only when the corresponding relation is valid for the simple transpositions $s_i$, namely for $\swap_i$ whenever $|\rho_i - \rho_{i+1}|>1$ and for $\braid_i$ whenever $\rho_{i-1}=\rho_{i+1} = \rho_i \pm 1$. The definition of $\mathfrak{d}_i$ clearly ensures this. The conditions for when to apply a swap or a braid are maintained by the swap or braid, ensuring that $\mathfrak{d}_i$ is an involution. Note that $\mathfrak{d}_i$ changes $\rho$ only at indices $i-1,i,i+1$, and the conditions that determine which swap or braid to apply look only at these indices. Therefore, if $|i-j|\geq 3$, then $\{i-1,i,i+1\}$ and $\{j-1,j,j+1\}$ are disjoint, the maps $\mathfrak{d}_i$ and $\mathfrak{d}_j$ commute.

  It remains to consider restricted dual equivalence classes under $\mathfrak{d}_h,\ldots,\mathfrak{d}_i$ for $i-h \leq 3$. For $i-h=0$, $\mathfrak{d}_i$ acts trivially if and only if $\rho_{i-1}\rho_{i}\rho_{i+1}$ is weakly increasing or weakly decreasing. Since, by the reduced condition, consecutive letters may not be equal, the sequence must be strict, and so the descent composition is $(3)$ (to get $s_{(3)}(X)$) or $(1,1,1)$ (to get $s_{(1,1,1)}(X)$). If $\mathfrak{d}_i$ acts nontrivially, then it pairs $acb$ with $cab$ (if $\mathfrak{d}_i = \swap_{i-1}$) or $bac$ with $bca$ (if $\mathfrak{d}_i = \swap_{i}$) or $aba$ with $bab$ (if $\mathfrak{d}_i = \braid_{i}$), where $a<b<c$. In all cases, the two descent compositions are $(2,1)$ and $(1,2)$ giving $s_{(2,1)}(X)$. A similar by hand analysis can handle the cases $i-h=1,2,3$, though it is more efficient (and perhaps less error-prone) to program the involutions on a computer and check them for all reduced words on $\{1,\ldots,2(i-h+3)\}$, since we need only separate the cases $|a-b|=0,1$ or $|a-b|>1$.
\end{proof}

A dual equivalence $\{\varphi_i\}$ for $(\mathcal{A},\Des)$ induces a $\Des$-preserving map $\Phi$ from $\mathcal{A}$ to $\SYT$ such that, for any dual equivalence class $\mathcal{C}$ for $\mathcal{A}$ under $\{\varphi_i\}$, the restriction of $\Phi$ to $\mathcal{C}$ gives a bijection with $\SYT(\lambda)$ for some (unique) partition $\lambda$.

\begin{definition}
  Given a dual equivalence $\{\varphi_i\}$ for $(\mathcal{A},\Des)$, we call the induced map $\Phi:\mathcal{A}\rightarrow\SYT$ the \emph{rectification map for $\mathcal{A}$ with respect to $\{\varphi_i\}$}. For $A \in\mathcal{A}$, we say that $A$ \emph{rectifies} to $\Phi(A)$.
  \label{def:rectification}
\end{definition}
  
For example, for reduced expressions for $42153$, the reduced expressions in Figure~\ref{fig:dual} rectify to the standard Young tableaux in Figure~\ref{fig:SYT-deg}, respectively.

\begin{definition}
  Given a permutation $w$, say that a reduced expression $\rho\in\R(w)$ is \emph{super-standard} if it rectifies to a super-standard tableau.
  \label{def:super-red}
\end{definition}
  
For example, the super-standard reduced expressions for $42153$ are shown in Figure~\ref{fig:super-red}.

\begin{figure}[ht]
  \begin{displaymath}
    \begin{array}{cccc@{\hskip 4em}ccc}
      (1,2,4,3,1) & \stackrel{\Phi}{\longrightarrow} & \raisebox{\cellsize}{$\tableau{5 \\ 4 \\ 1 & 2 & 3}$} & &
      (1,2,4,1,3) & \stackrel{\Phi}{\longrightarrow} & \tableau{4 & 5 \\ 1 & 2 & 3}
    \end{array}
  \end{displaymath}
  \caption{\label{fig:super-red}The super-standard elements of $\R(42153)$ and their rectifications.}
\end{figure}

In particular, combining Theorems~\ref{thm:deg-red} and \ref{thm:positivity} with Proposition~\ref{prop:super-SYT}, we have a simple, combinatorial proof of the following.

\begin{corollary}
  For $w$ a permutation, we have
  \begin{equation}
    \stanley_w = \sum_{\substack{\rho\in\R(w) \\ \rho \ \mathrm{super-standard}}} s_{\Des(\rho)} = \sum_{\lambda} c_{w,\lambda} s_{\lambda},
  \end{equation}
  where $c_{w,\lambda}$ is the number of super-standard reduced expressions for $w$ with descent composition $\lambda$.
  \label{cor:stanley-pos}
\end{corollary}

%
\section{Key positivity of Schubert polynomials}
%
\label{sec:polynomials}

\subsection{Schubert polynomials}
\label{sec:schubert}

Lascoux and Sch{\"u}tzenberger \cite{LS82} defined polynomial representatives for the Schubert classes in the cohomology ring of the complete flag variety. The importance of these \emph{Schubert polynomials} lies in the fact that their structure constants give intersection multiplicities for the corresponding varieties. Finding a combinatorial rule to compute these number remains one of the fundamental open problems in algebraic combinatorics. We refer the reader to \cite{Mac91} for a beautiful and thorough treatment of the underlying combinatorics of Schubert polynomials, insofar as it is understood.

As with the Schur case, we will harness the power of another basis, in this case the \emph{fundamental slide basis} \cite{AS17} of Assaf and Searles, to express Schubert polynomials as the generating function for reduced expressions.

\begin{definition}[\cite{AS17}]
  For a weak composition $a$ of length $n$, define the \emph{fundamental slide polynomial} $\fund_{a} = \fund_{a}(x_1,\ldots,x_n)$ by
  \begin{equation}
    \fund_{a} = \sum_{\substack{b \geq a \\ \flatten(b) \ \mathrm{refines} \ \flatten(a)}} x_1^{b_1} \cdots x_n^{b_n},
    \label{e:fundamental-shift}
  \end{equation}
  where $b \geq a$ means $b_1 + \cdots + b_k \geq a_1 + \cdots + a_k$ for all $k=1,\ldots,n$.
  \label{def:fundamental-shift}
\end{definition}

For example, we have
\begin{displaymath}
  \fund_{(0,3,2)} = x^{032} + x^{122} + x^{212} + x^{302} + x^{311} + x^{320}.
\end{displaymath}

Whereas fundamental quasisymmetric functions are indexed by strong compositions, fundamental slide polynomials are indexed by weak compositions, so we require a weak descent composition to define generating functions with respect to this basis. We adopt the following from \cite{Ass-1}.

\begin{definition}[\cite{Ass-1}]
    For a reduced expression $\rho$ with run decomposition $(\rho^{(k)} | \cdots | \rho^{(1)})$, set $r_k = \rho^{(k)}_1$ and, for $i<k$, set $r_i = \min(\rho^{(i)}_1,r_{i+1}-1)$. Define the \emph{weak descent composition of $\rho$}, denoted by $\des(\rho)$, by $\des(\rho)_{r_i} = |\rho^{(i)}|$ and all other parts are zero if all $r_i>0$ and $\des(\rho) = \varnothing$ otherwise.
  \label{def:des-red}
\end{definition}

For example, among the reduced expressions in Figure~\ref{fig:red}, all but the first and fourth in the top row are virtual, and these have weak descent compositions $(3,1,0,1)$ and $(3,2,0,0)$, respectively. To facilitate virtual objects, we extend notation and set
\begin{equation}
  \fund_{\varnothing} = 0.
\end{equation}

Building on the monomial model given by Billey, Jockusch, and Stanley \cite{BJS93}, Assaf \cite{Ass-1} gave the following expansion of Schubert polynomials in terms of fundamental slide polynomials, which we take as our definition.

\begin{definition}[\cite{Ass-1}]
  For $w$ any permutation, we have
  \begin{equation}
    \schubert_{w} = \sum_{\rho \in \R(w)} \fund_{\des(P)},
    \label{e:schubert-slide}
  \end{equation}
  where the sum may be taken over non-virtual reduced expressions $\rho$.
  \label{def:schubert-slide}
\end{definition}

There is a special case worth mentioning, that of \emph{grassmannian permutations} which are permutations with at most one descent. Given a partition $\lambda$ of length $j$ and a positive integer $k \geq j$, the \emph{grassmannian permutation associated to $\lambda$ and $k$}, denoted by $v(\lambda,k)$, is given by
\begin{equation}
  v(\lambda,k)_i = i + \lambda_{k-i+1}
  \label{e:grass}
\end{equation}
for $i = 1,\ldots,k$, where we take $\lambda_i=0$ for $i>j$, and $v(\lambda,k)$ has a unique descent at $k$. For example,
\begin{displaymath}
  \begin{array}{rcrrrrrrrrrrr}
    & & 0 & 0 & 1 & 4 & 4 & 5 \ \vline & & & & & \\\cline{3-13}
    v((5,4,4,1), 6) & = & 1 & 2 & 4 & 8 & 9 & 1\!1 \ \vline & 3 & 5 & 6 & 7 & 1\!0 . 
  \end{array}
\end{displaymath}
It is easy to see that $v(\lambda,k)$ gives a bijection between grassmannian permutations with unique descent at $k$ and partitions of length at most $k$. Moreover, we have the following.

\begin{theorem}[\cite{LS82}]
  For $\lambda$ a partition and $k$ a positive integer, we have
  \begin{equation}
    \schubert_{v(\lambda,k)} = s_{\lambda}(x_1,\ldots,x_k).
  \end{equation}
\end{theorem}

Therefore the Schubert polynomials contain the Schur polynomials as a special case. However, we argue that Schubert polynomials more closely parallel Stanley symmetric functions than they do Schur functions, noting that the latter are also a special case of the former.

Let $1^m \times w$ denote the permutation obtained by adding $m$ to all values of $w$ in one-line notation and pre-pending $1,2,\ldots,m$. Note that the reduced expressions for $1^m \times w$ are simply those for $w$ with each index increased by $m$. Let $0^m \times a$ denote the weak composition obtained by pre-pending $m$ zeros to $a$. Then for $\rho \in R(w)$ non-virtual, the corresponding reduced expression for $R(1^m \times w)$ will have weak descent composition $0^m \times \des(\rho)$. To make our running example slightly more interesting, consider $1 \times 42153 = 153264$. From Figure~\ref{fig:red-nv}, we have
\begin{displaymath}
  \schubert_{153264} = \fund_{(0,3,1,0,1)} + \fund_{(2,2,0,0,1)} + \fund_{(1,3,0,0,1)} + \fund_{(0,3,2,0,0)} + \fund_{(2,2,1,0,0)} + \fund_{(1,3,1,0,0)} + \fund_{(2,3,0,0,0)}.
\end{displaymath}

\begin{figure}[ht]
  \begin{displaymath}
    \begin{array}{ccccccc}
      (5,3,2,3,4) & (5,2,3,2,4) & (5,2,3,4,2) & (3,5,2,3,4) & (3,2,5,3,4) & (3,2,3,5,4) & (2,3,5,2,4) 
    \end{array}
  \end{displaymath}
   \caption{\label{fig:red-nv}The set of non-virtual reduced expressions for $153264$.}
\end{figure}

Note that the fundamental slide expansion of $\schubert_{1^m \times 42153}$ gains no additional terms when $m$ is at least two. Macdonald \cite{Mac91} explained this stability phenomenon by showing that Stanley symmetric functions are the \emph{stable limits} of Schubert polynomials. 

\begin{proposition}[\cite{Mac91}]
  For $w$ a permutation, we have
  \begin{equation}
    \lim_{m \rightarrow \infty} \schubert_{1^m \times w} = \stanley_{w}(X).
  \end{equation}
  \label{prop:schub-limit}
\end{proposition}

In parallel to this, Assaf and Searles \cite{AS17} showed that fundamental quasisymmetric functions are the stable limits of fundamental slide polynomials.

\begin{proposition}[\cite{AS17}]
  For a weak composition $a$, we have
  \begin{equation}
    \lim_{m \rightarrow \infty} \fund_{0^m \times a} = F_{\flatten(a)}(X).
  \end{equation}
  \label{prop:fund-limit}
\end{proposition}

Therefore flattening the strong compositions in the fundamental slide expansion of $\schubert_{1^m \times w}$ precisely gives the fundamental quasisymmetric expansion of $\stanley_w$. To fit Schur functions into this stable picture, we consider another basis for the polynomial ring called \emph{key polynomials}.

\subsection{Key polynomials}
\label{sec:key}

The key polynomials first arose as Demazure characters for the general linear group \cite{Dem74} and were later studied combinatorially by Lascoux and Sch{\"u}tzenberger \cite{LS90} who expounded on their connection with Schubert polynomials. As with Schubert polynomials, original definitions were given in terms of divided differences, though we will derive a combinatorial model in terms of fundamental slide polynomials based on work of Kohnert \cite{Koh91} and Assaf and Searles \cite{AS-2}. See \cite{RS95} for a thorough treatment of the combinatorics of key polynomials.

A \emph{diagram} is a finite collection of cells in the $\mathbb{Z}\times\mathbb{P}$ lattice. We index each cell of a diagram by its top right corner. A diagram is \emph{virtual} if it contains a cell with nonpositive row index.

\begin{figure}[ht]
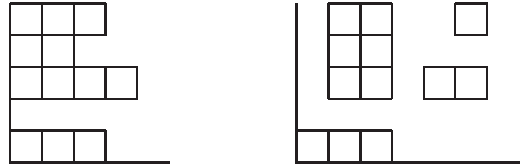

  \begin{displaymath}
    \vline\tableau{ \ & \ & \ \\ \ & \ \\  \ & \ & \ & \ & \\ \\ \ & \ & \ \\\hline }
    \hspace{4\cellsize}
    \vline\tableau{ & \ & \ & & & \ \\ & \ & \ \\ & \ & \ & & \ & \ & \\ \\ \ & \ & \ \\\hline }
  \end{displaymath}
   \caption{\label{fig:diagrams}The key diagram for $(3,0,4,2,3)$ (left) and another diagram with the same weight (right).}
\end{figure}

The \emph{weight} of a diagram $D \subset \mathbb{P} \times \mathbb{P}$, denoted by $\wt(D)$, is the weak composition whose $i$th part is the number of cells in row $i$ of $D$. The weight of a virtual diagram is $\varnothing$.

\begin{definition}[\cite{AS-2}]
  Given a weak composition $a$ of length $n$, a \emph{Kohnert tableau of shape $a$} is a diagram filled with entries $1^{a_1}, 2^{a_2}, \ldots, n^{a_n}$, one per cell, satisfying the following conditions:
  \begin{enumerate}[label=(\roman*)]
  \item there is exactly one $i$ in each column $1$ through $a_i$;
  \item each entry in row $i$ is at least $i$;
  \item the $i$'s weakly descend from left to right;
  \item if $i<j$ appear in a column with $i$ above $j$, then there is an $i$ right of and strictly above $j$.
  \end{enumerate}
  \label{def:kohnert}
\end{definition}

Kohnert tableaux are so named because they are based on \emph{Kohnert moves} on \emph{key diagrams}. The \emph{key diagram} of a weak composition $a$, denoted by $\YD(a)$, is the diagram with $a_i$ cells left-justified in row $i$. For example, the left diagram in Figure~\ref{fig:diagrams} is the key diagram for $(3,0,4,2,3)$. A \emph{Kohnert move} on a diagram selects a nonempty row and the rightmost cell therein, then pushes this cell down to the highest empty space below it. Kohnert \cite{Koh91} showed that set of the diagrams obtained from Kohnert moves on a key diagram generate the key polynomial. Assaf and Searles \cite{AS-2} showed that these diagrams are in bijection with Kohnert tableaux, the latter being easier to enumerate directly. 

\begin{definition}[\cite{AS-2}]
  A Kohnert tableau is \emph{quasi-Yamanouchi} if each nonempty row $i$ either has an entry equal to $i$ or has a cell weakly left of some cell in row $i+1$. 
  \label{def:quasi-Yam}
\end{definition}

Denote the set of quasi-Yamanouchi Kohnert tableaux of shape $a$ by $\QKT(a)$.

\begin{definition}[\cite{AS-2}]
  For a weak composition $a$, we have
  \begin{equation}
    \key_a = \sum_{D \in \QKT(a)} \fund_{\wt(D)},
    \label{e:key-slide}
  \end{equation}
  where the sum may be taken over all non-virtual quasi-Yamanouchi Kohnert tableaux.
  \label{def:key-slide}
\end{definition}

\begin{figure}[ht]
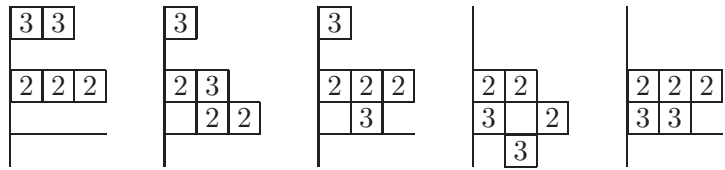

  \begin{center}
    \begin{displaymath}
      \begin{array}{c@{\hskip 2em}c@{\hskip 2em}c@{\hskip 2em}c@{\hskip 2em}c}
        \vline\tableau{ 3 & 3 \\ \\ 2 & 2 & 2 \\ \\ \hline \\ } &
        \vline\tableau{ 3 \\ \\ 2 & 3 \\ & 2 & 2 \\ \hline \\ } &
        \vline\tableau{ 3 \\ \\ 2 & 2 & 2 \\ & 3 \\ \hline \\ } &
        \vline\tableau{ \\ \\ 2 & 2 &  \\ 3 & & 2 \\ \hline & 3  \\ } &
        \vline\tableau{ \\ \\ 2 & 2 & 2 \\ 3 & 3 \\ \hline \\ } 
      \end{array}
    \end{displaymath}
    \caption{\label{fig:qYam_kohnert_tableaux}Quasi-Yamanouchi Kohnert tableaux of shape $(0, 3, 0, 2)$.}
  \end{center}
\end{figure}

For example, Figure~\ref{fig:qYam_kohnert_tableaux} gives $\QKT(0,3,0,2)$. From this we compute
\begin{displaymath}
 \key_{(0,3,0,2)} = \fund_{(0,3,0,2)} + \fund_{(2,2,0,1)} + \fund_{(1,3,0,1)} + \fund_{(2,3,0,0)}.
\end{displaymath}

We may reverse Kohnert moves on quasi-Yamanouchi Kohnert tableaux to give a simple tableau model for key polynomials in terms of certain fillings of key diagrams.

\begin{definition}
  A \emph{standard key tableau} is a bijective filling of a key diagram with $\{1,2,\ldots,n\}$ such that rows weakly decrease and if some entry $i$ is above and in the same column as an entry $k$ with $i<k$, then there is an entry immediately right of $k$, say $j$, and $i<j$. 
  \label{def:key-tab}
\end{definition}

We denote the set of key tableaux of shape $a$ by $\Key(a)$. For example, see Figure~\ref{fig:key-tableau}. 

\begin{figure}[ht]
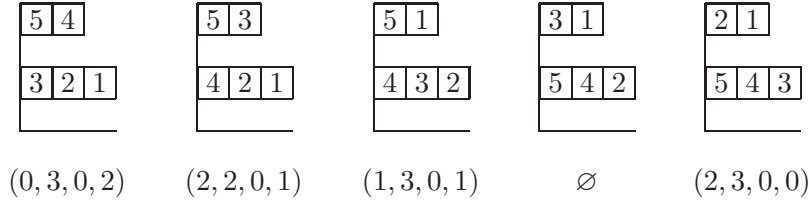

  \begin{displaymath}
    \begin{array}{c@{\hskip 2em}c@{\hskip 2em}c@{\hskip 2em}c@{\hskip 2em}c}
      \vline\tableau{5 & 4 \\ \\ 3 & 2 & 1 \\ \\\hline} &
      \vline\tableau{5 & 3 \\ \\ 4 & 2 & 1 \\ \\\hline} &
      \vline\tableau{5 & 1 \\ \\ 4 & 3 & 2 \\ \\\hline} &
      \vline\tableau{3 & 1 \\ \\ 5 & 4 & 2 \\ \\\hline} &
      \vline\tableau{2 & 1 \\ \\ 5 & 4 & 3 \\ \\\hline} \\ \\
      (0,3,0,2) & (2,2,0,1) & (1,3,0,1) & \varnothing & (2,3,0,0) 
    \end{array}
  \end{displaymath}
  \caption{\label{fig:key-tableau}Standard key tableaux of shape $(0,3,0,2)$ and their weak descent compositions.}
\end{figure}

\begin{definition}
  For a standard tableau $T$, the \emph{run decomposition of $T$} is $\tau = (\tau^{(k)}|\ldots|\tau^{(1)})$, where $\tau$ is the decreasing word $n \cdots 2 1$ broken between $i+1$ and $i$ precisely when $i+1$ lies weakly right of $i$ in $T$. In this case, we call $i$ a \emph{descent} of $T$.
  \label{def:run-std}
\end{definition}

For example, the run decompositions for the standard key tableaux in Figure~\ref{fig:key-tableau} are $(54|321)$, $(5|43|21)$, $(5|432|1)$, $(54|32|1)$, $(543|21)$, respectively. 

\begin{definition}
  For a standard tableau $T$, let $(\tau^{(k)}|\ldots|\tau^{(1)})$ be the run decomposition of $T$. Set $t_k = \mathrm{row}(\tau^{(k)}_1)$ and, for $i<k$, set $t_i = \min(\mathrm{row}(\tau^{(i)}_j),t_{i+1}-1)$, where $j=1,\ldots,|\tau^{(i)}|$. Define the \emph{weak descent composition of $T$}, denoted by $\des(T)$, by $\des(T)_{t_i} = |\tau^{(i)}|$ and all other parts are zero if all $t_i>0$ and $\des(T) = \varnothing$ otherwise.
  \label{def:des-std}
\end{definition}

\begin{remark}
  Note that, for the current case of standard key tableaux, it is enough to take $t_i = \min(\mathrm{row}(\tau^{(i)}_{1}),t_{i+1}-1)$. To see why, if $k+1,k$ are in $\tau^{(i)}$ with $k$ in row $r$ strictly below, then since $k$ is strictly right of $k+1$, there must be some larger entry, say $\ell$, left of $k$ in its row and in the column of $k+1$. In this case, $\ell$ must be in $\tau^{j}$ for some $j>i$ and we must have $t_j \leq r$. Therefore $t_i$ will not attain its value at $r$, the row of $k$. Despite this apparent simplification, we keep this more general definition for weak descent compositions as it is needed in \S~\ref{sec:LRR}.
\end{remark}
  
For example, weak descent compositions for the standard key tableaux in Figure~\ref{fig:key-tableau} are shown.

\begin{definition}
  For $D\in\QKT(a)$, the \emph{ascended tableau of $D$}, denoted by $\ascend(D)$, is obtained by re-labeling the cells of $D$ along rows from left to right, beginning at the top, with $n,n-1,\ldots,2,1$, and returning any cell originally labeled by $i$ back to row $i$.

  For $T\in\Key(a)$, the \emph{descended diagram of $T$}, denoted by $\descend(T)$, is the diagram obtained by pushing cells down minimally until the word obtained by reading entries right to left, from bottom to top is the identity, and then relabeling cells based on their original row index.
  \label{def:key-descend}
\end{definition}

For example, the ascended tableaux for the quasi-Yamanouchi Kohnert tableaux in Figure~\ref{fig:qYam_kohnert_tableaux} are given in Figure~\ref{fig:key-tableau}, respectively. Conversely, the descended diagrams for the standard key tableaux in Figure~\ref{fig:key-tableau} are shown in Figure~\ref{fig:qYam_kohnert_tableaux}, respectively. 


\begin{theorem}
  The maps $\ascend$ and $\descend$ are inverse bijections between $\Key(a)$ and $\QKT(a)$ such that $\wt(D) = \des(\ascend(D))$ and $\des(T) = \wt(\descend(T))$.
  \label{thm:key-descend}
\end{theorem}

\begin{proof}
  Consider first $\ascend(D)$ for $D$ a Kohnert tableau for $a$. Condition (i) of Definition~\ref{def:kohnert} ensures that $\ascend(D)$ is a labeling of the key diagram of $a$; condition (ii) ensures that cells move weakly up in passing from $D$ to $\ascend(D)$; condition (iii) ensures that rows of $\ascend(D)$ are weakly decreasing; and condition (iv) ensures that if some entry $i$ of $\ascend(D)$ lies above some entry $k$ with $i<k$ in the same column, then there is any entry $j$ immediately right of $k$ with $i<j$. In particular, $\ascend(D) \in \Key(a)$. 

  Next consider $\descend(T)$ for $T\in\Key(a)$. Re-labeling cells based on original row index ensures condition (i) of Definition~\ref{def:kohnert} holds for $\descend(T)$; cells moving down ensures condition (ii); entries weakly decreasing along rows ensures that cells to the right in the same row move weakly lower, giving condition (iii); and the column inversion condition for key tableaux precisely corresponds to condition (iv). Therefore $\descend(T)$ is a Kohnert tableau for $a$. Moreover, for every nonempty row $r$ of the key diagram for $a$, either the leftmost entry, say $i$, remains in row $r$ or it must be pushed minimally down, in which case sits in the row immediately below $i+1$ which, by the definition of descents, lies weakly to its right. Thus $\descend(T)$ is quasi-Yamanouchi.

  With images established, the maps are clearly inverse to one another, proving that both are indeed bijections. For $D\in\QKT(a)$, after we re-label cells, $i$ will be a descent of $\ascend(D)$ if and only if $i$ is the leftmost in its row. In particular, $\wt(D) = \des(\ascend(D))$.
\end{proof}

In particular, standard key tableaux give another characterization of key polynomials.

\begin{corollary}
  The key polynomial for a weak composition $a$ is given by
  \begin{equation}
    \key_a = \sum_{T \in \Key(a)} \fund_{\des(T)},
    \label{e:key-fund-des}
  \end{equation}
  where the sum may be taken over non-virtual standard key tableaux of shape $a$.
\end{corollary}

Given a partition $\lambda$ and a positive integer $k$ that is at least the length of $\lambda$, let $a(\lambda,k)$ denote the weak composition of length $k$ with weakly increasing parts that sort to $\lambda$. Then we have the following.

\begin{corollary}
  For $\lambda$ a partition and $k$ a positive integer, we have
  \begin{equation}
    \key_{a(\lambda,k)} = s_{\lambda}(x_1,\ldots,x_k).
  \end{equation}
\end{corollary}

Therefore the key polynomials also contain the Schur polynomials as a special case. We argue that the parallel here is much deeper than with Schubert polynomials. 

We say that a standard key tableau is \emph{yamanouchi} if its reverse row reading word is the identity. We have the following key tableau analog of Proposition~\ref{prop:super-SYT}.

\begin{proposition}
  For a weak composition $a$, there exists a unique yamanouchi key tableau $Y$ of shape $a$. Moreover, for $T\in\Key(a)$, we have $\des(T) = a$ if and only if $T=Y$.
  \label{prop:super-Key}
\end{proposition}

\begin{proof}
  Clearly $Y$ is unique, if it exists, and it can be constructed by placing entries $1,\ldots,a_1$ left to right in row $1$, $a_1+1,\ldots,a_1+a_2$ left to right in row $2$, and so on. Then $i$ will be a descent of $Y$ precisely for the partial sums $i>0$ in $\{a_1,a_1+a_2,\ldots\}$, giving $\des(Y)=a$. Conversely, if $\Des(T)=a$, then the upper uni-triangularity of key polynomials with respect to monomials evident from Kohnert's expansion shows that $\descend(T)$ is the key diagram for $a$, in which case $T=Y$.
\end{proof}

Standard key tableaux provide the natural analog for standard Young tableaux in our generalization of dual equivalence.

Comparing key tableaux for $a$ with those for $0^m \times a$, prepending $0$'s simply prepends $0$'s to the weak descent composition. Note that the number of terms in the fundamental slide expansion of $\key_{0^m \times (3,2)}$ remains the same when $m$ is at least two. This stability phenomenon is explained by the fact that Schur functions are the \emph{stable limits} of key polynomials. This is implicit in work of Lascoux and Sch{\"u}tzenberger and is made explicit in \cite{AS-2}.

\begin{proposition}
  For a weak composition $a$, we have
  \begin{equation}
    \lim_{m \rightarrow \infty} \key_{0^m \times a} = s_{\sort(a)}(X).
  \end{equation}
  \label{prop:key-limit}
\end{proposition}

Furthermore, Assaf and Searles \cite{AS-2} proved that flattening the compositions in the fundamental slide expansion of $\key_{0^m \times a}$ precisely gives the fundamental quasisymmetric expansion of $s_{\sort(a)}$. For example, flatten $\key_{0^2 \times (3,2)}$ and compare with $s_{(3,2)}$.

Lifting the Schur positivity of Stanley symmetric functions, Schubert polynomials are known to expand nonnegatively in the key basis \cite{LS90}. For example, we have
\begin{displaymath}
  \schubert_{42153} = \key_{(3,1,0,1)} + \key_{(3,2,0,0)}.
\end{displaymath}
As with Schur functions and Stanley symmetric functions, we give an independent and elementary proof of this by lifting dual equivalence to polynomials.

\subsection{Weak dual equivalence}
\label{sec:keys}

Given a set of combinatorial objects $\mathcal{A}$ endowed with a notion of weak descents, one can form the fundamental slide generating polynomial for $\mathcal{A}$ by
\begin{displaymath}
  \sum_{T \in \mathcal{A}} \fund_{\des(T)} .
\end{displaymath}
Two examples of this are Schubert polynomials generated by reduced expressions \eqref{e:schubert-slide} and key polynomials generated by standard key tableaux \eqref{e:key-fund-des}. We generalize the notion of dual equivalence to polynomials defined in this way as follows.

\begin{definition}
  Let $\mathcal{A}$ be a finite set, and let $\des$ be a map from $\mathcal{A}$ to weak compositions of $n$. A \emph{weak dual equivalence for $(\mathcal{A},\des)$} is a family of involutions $\{\psi_i\}_{1<i<n}$ on $\mathcal{A}$ such that
  \renewcommand{\theenumi}{\roman{enumi}}
  \begin{enumerate}
  \item For all $i-h \leq 3$ and all $T \in \mathcal{A}$, there exists a weak composition $a$ of $i-h+3$ such that
    \[ \sum_{U \in [T]_{(h,i)}} \fund_{\des_{(h-1,i+1)}(U)} = \key_{a}, \]
    where $[T]_{(h,i)}$ is the equivalence class generated by $\psi_h,\ldots,\psi_i$, and $\des_{(h,i)}(T)$ is the weak composition of $i-h+1$ obtained by deleting the first $h-1$ and last $n-i$ nonzero parts from $\des(T)$.
    
  \item For all $|i-j| \geq 3$ and all $T \in\mathcal{A}$, we have
    \begin{displaymath}
      \psi_{j} \psi_{i}(T) = \psi_{i} \psi_{j}(T).
    \end{displaymath}

  \end{enumerate}

  \label{def:deg-weak}
\end{definition}

For example, if $\des(T) = (0,3,2,0,3,1)$, then $\des_{(3,8)}(T) = (0,1,2,0,3,0)$.

As a first example of weak dual equivalence, we construct a weak dual equivalence for standard key tableaux. Define the column reading order of a standard key tableau to begin at the lowest cell of the leftmost column, read entries in the column bottom to top, then continue with the next column to the right. For example, the column reading order for the leftmost tableau in Figure~\ref{fig:key-tableau} is $35241$.

\begin{definition}
  Given $T \in \Key(a)$ and $1 < i < |a|$, define $d_i(T)$ as follows. Let $b,c,d$ be the cells with entries $i-1,i,i+1$ taken in column reading order. Then
  \begin{equation}
    d_i (T) = \left\{ \begin{array}{rl}
      \braid_{i}  (T) & \mbox{if $b,d$ are in the same row and $c$ is not} , \\
      \swap_{i-1}(T) & \mbox{else if $c$ has entry $i+1$} , \\
      \swap_{i}  (T) & \mbox{else if $c$ has entry $i-1$} , \\
      T & \mbox{otherwise},
    \end{array} \right.
  \end{equation}
  where $\braid_{j}$ cycles $j-1,j,j+1$ so that $j$ shares a row with $j \pm 1$ and $\swap_j$ interchanges $j$ and $j+1$.
  \label{def:key-deg}
\end{definition}

\begin{figure}[ht]
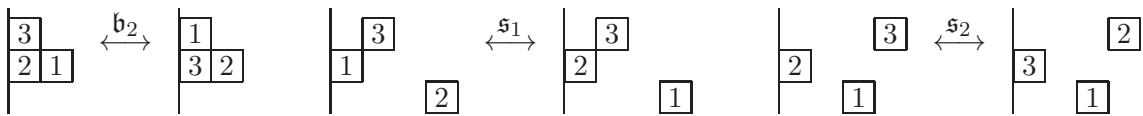

  \begin{displaymath}
    \begin{array}{ccc@{\hskip 3em}ccc@{\hskip 3em}ccc}
      \vline\tableau{ 3 & \\ 2 & 1} &
      \stackrel{\displaystyle\braid_{2}} \longleftrightarrow &
      \vline\tableau{ 1 & \\ 3 & 2} &
      \vline\tableau{ & 3 & & \\ 1 \\ & & & 2} &
      \stackrel{\displaystyle\swap_{1}} \longleftrightarrow &
      \vline\tableau{ & 3 & & \\ 2 \\ & & & 1} &
      \vline\tableau{ & & & 3 \\ 2 \\ & & 1 & } &
      \stackrel{\displaystyle\swap_{2}} \longleftrightarrow &
      \vline\tableau{ & & & 2 \\ 3 \\ & & 1 & } 
    \end{array}
  \end{displaymath}
  \caption{\label{fig:dual-SKT}The dual equivalence map $d_i$ on standard key tableaux.}
\end{figure}

See Figure~\ref{fig:dual-SKT} for a graphical representation of $d_i$, and see Figure~\ref{fig:key-deg} for examples.

\begin{lemma}
  The maps $\{d_i\}$ are well-defined involutions on $\Key(a)$. 
  \label{lem:dual-SKT}  
\end{lemma}

\begin{proof}
  Let $T\in\Key(a)$ and, with notation as in Definition~\ref{def:key-deg}, suppose $b,d$ are in the same row and $c$ is not. We use the key tableaux condition that if $i<k$ are in a column with $i$ above $k$, then there is a $j>i$ immediately right of $k$ to argue that we have the case depicted in the left of Figure~\ref{fig:dual-SKT}. First we claim that $c$ does not have value $i$. If it did, then $b$ must have value $i+1$ and $d$ value $i-1$, in which case they must be adjacent in their row. Since $i$ lies between in column reading order, it must lie above $i+1$, a contraction of the key tableaux condition since $i<i+1$ but $i>i-1$, or below $i-1$, a contradiction again since anything to the right of $i$ must be smaller than $i-1$. Therefore $c$ is either $i-1$, in which case it cannot be below $i$ since the entry to its left will be below and larger than $i+1$, or $c$ is $i+1$, in which case it cannot sit below $i-1$ since any entry to its right will be below and smaller than $i-1$. Thus we have $c$ above $b$ and $d_i(T) = \braid_i(T)$ acts as shown in Figure~\ref{fig:dual-SKT}.

  Next suppose $b$ and $d$ lie in different rows. The same key tableaux condition ensures that if $c$ does not have entry $i$, then $b$ and $d$ must also lie in different columns. In this case, all entries in the same row or column as $b$ or $d$ compare the same with $i$ and $i\pm 1$, so the key tableaux conditions are preserved by whichever of $\swap_{i-1}$ or $\swap_i$ applies. 
\end{proof}

\begin{definition}
  For a standard key tableau $T$ with run decomposition $(\tau^{(k)}|\ldots|\tau^{(1)})$, the \emph{descent composition of $T$}, denoted by $\Des(T)$, is the strong composition $\Des(T) = (|\tau^{(1)}|,\ldots,|\tau^{(k)}|)$. 
  \label{def:key-Des}
\end{definition}

Again, notice the reversal of lengths. Note also that $\Des(T) = \flatten(\des(T))$.

It is clear that inserting or deleting rows with no cells does not change the allowable fillings of the standard key tableaux. That is, there is an obvious $\Des$-preserving bijection between $\Key(a)$ and $\Key(b)$ whenever $\flatten(a) = \flatten(b)$. The $\Des$-preserving bijection between $\Key(a)$ and $\Key(b)$ whenever $\sort(a) = \sort(b)$ is less obvious. It follows as a corollary to the following.

\begin{theorem}
  The maps $\{d_i\}$ give a dual equivalence for $(\Key(a),\Des)$ consisting of a single equivalence class. In particular, we have
  \begin{equation}
    s_{\lambda} = \sum_{T \in \Key(a)} F_{\Des(T)},
  \end{equation}
  for any partition $\lambda$ and weak composition $a$ such that $\sort(a) = \lambda$.
  \label{thm:dual-SKT}
\end{theorem}

\begin{proof}
  Consider the map $\Phi$ on $\Key(a)$ defined by letting the cells of $T$ fall to shape $\sort(a)$, then sorting the columns to descend upward, and replacing $i$ with $n-i+1$. After letting entries fall and sorting columns, the rows will necessarily be decreasing left to right, and so $\Phi(T)\in\SYT(\sort(a))$. Note that $i+1$ lies strictly right of $i$ in $T$ if and only if $n-i$ lies strictly right of $n-i+1$ in $\Phi(T)$. In particular, $i\in\Des(T)$ if and only if $n-i\in\Des(\Phi(T))$.

  Moreover, we claim that $\Phi(d_i(T)) = d_{n-i+1}(\Phi(T))$. To see this, note that, by the ever useful key tableaux condition, if $i$ and $i+1$ appear in the same column of $T$, $i+1$ must be above $i$ since anything to the right of $i+1$ is smaller than $i$. Therefore the column reading word for $T$ restricted to $i-1,i,i+1$ maps to that for $\Phi(T)$ restricted to $n-i,n-i+1,n-i+2$. Then $\swap_{i-1}$ and $\swap_i$ correspond under $\Phi$ to exchanging the larger two and smaller two entries, respectively, and, upon sorting columns, $\braid_i$ corresponds to swapping the smaller two, thus establishing the claim. 

  The theorem now follows from the symmetry of Schur functions, and the corollary from the super-standard characterization of tableaux.
\end{proof}

\begin{figure}[ht]
  \begin{displaymath}
    \begin{array}{ccccccccc}
      \vline\tableau{5 & 4 \\ \\ 3 & 2 & 1 \\ \\\hline} &
      \raisebox{-\cellsize}{$\stackrel{\displaystyle\stackrel{d_3}{\Longleftrightarrow}}{\scriptstyle d_4}$}  &
      \vline\tableau{5 & 3 \\ \\ 4 & 2 & 1 \\ \\\hline} &
      \raisebox{-0.5\cellsize}{$\stackrel{d_2}{\longleftrightarrow}$} &
      \vline\tableau{5 & 1 \\ \\ 4 & 3 & 2 \\ \\\hline} &
      \raisebox{-0.5\cellsize}{$\stackrel{d_4}{\longleftrightarrow}$} &
      \vline\tableau{3 & 1 \\ \\ 5 & 4 & 2 \\ \\\hline} &
      \raisebox{-\cellsize}{$\stackrel{\displaystyle\stackrel{d_2}{\Longleftrightarrow}}{\scriptstyle d_3}$}  &
      \vline\tableau{2 & 1 \\ \\ 5 & 4 & 3 \\ \\\hline}
    \end{array}
  \end{displaymath}
  \caption{\label{fig:key-deg}Dual equivalence for $\Key(0,3,0,2)$.}
\end{figure}

\begin{theorem}
  Given a weak composition $a$, the involutions $d_i$ give a weak dual equivalence for $(\Key(a),\des)$ consisting of a single equivalence class. 
  \label{thm:deg-key}
\end{theorem}

\begin{proof}
  The action of $d_i$ on $T\in\Key(a)$ is completely determined by the positions $i+1,i,i-1$, and the relative positions of cells other than these remains unchanged under $d_i$. Therefore, if $|i-j|\geq 3$, then $\{i-1,i,i+1\}$ and $\{j-1,j,j+1\}$ are disjoint, the maps $d_i$ and $d_j$ commute. It remains only to consider restricted equivalence classes under $d_h,\ldots,d_i$ for $i-h \leq 3$.

  Consider the case $i-h=0$. From the proof of Lemma~\ref{lem:dual-SKT}, we have that $d_i(T)=T$ if and only if $c=i$, in which case either both or neither of $i-1,i$ is a descent of $T$, so the restricted weak descent composition flattens to either $(1,1,1)$ or $(3)$. In either case, the corresponding key polynomial is a single fundamental slide polynomial, and so the equivalence class corresponds to a single key polynomial. If $d_i(T)=\braid_i(T)$, then we may assume $T$ has $c=i+1$. Then the restricted run decomposition of $T$ is $i\!+\!1 | i i\!-\!1$, and that of $\braid_i(T)$ is $i\!+\!1 i | i\!-\!1$. Therefore $\des_{(i-1,i+1)}(T) = (0^m,2,0^n,1)$, where $n=0$ if and only if either $i+1$ lies in the row immediately above $i$ and $i-1$ or if $t_j$ of the block containing $i+1$ is forced to be $t_{j+1}-1$. Either way, we have $\des_{(i-1,i+1)}(\braid_i(T)) = (0^{m-1},1,2)$, and so the equivalence class corresponds to the polynomial $\fund_{(0^{m},2,0^{n},1)} + \fund_{(0^{m-1},1,2)} = \key_{(0^{m},2,0^{n},1)}$.  If $d_i(T)=\swap_{i-1}(T)$, and so, again, $c=i+1$. We may assume $T$ has $i-1$ left of $i$. Then the restricted run decomposition of $T$ is $i\!+\!1 i | i\!-\!1$, and that of $\swap_{i-1}(T)$ is $i\!+\!1 | i i\!-\!1$. Moreover, in this case, for both $T$ and $\swap_{i-1}(T)$, both blocks of the run decomposition must be forced to have $t_j = t_{j+1}-1$ since there is an entry larger than $i$ left of $i$ and below $i-1$. Therefore $\des_{(i-1,i+1)}(T) = (0^m,1,2)$ and $\des_{(i-1,i+1)}(\swap_{i-1}(T)) = (0^{m},2,1)$, and so the equivalence class corresponds to the polynomial $\fund_{(0^{m},1,2)} + \fund_{(0^{m},2,1)} = \key_{(0^{m},1,2)}$. The argument for $d_i(T)=\swap_{i}(T)$ is completely analogous.

  Again, one can either carry out similar analyses for the cases $i-h=1,2,3$, or, to avoid tedium, since the action of $d_i$ is determined by relative positions of these cells based on their rows and columns, there are finitely many configurations to check by computer.
\end{proof}

Under certain stability assumptions, the converse of Theorem~\ref{thm:deg-key} also holds.

\begin{definition}
  The key polynomial $\key_a$ is \emph{$\fund$-stable} if both $\key_a$ and $\key_{0^m \times a}$ have the same number of terms in their $\fund$-expansions for any $m>0$.
\end{definition}

To make this condition easier to establish, we have the following proposition proved in \cite{AS-2}.

\begin{proposition}
  The following conditions are equivalent:
  \begin{enumerate}
  \item $\Key(a)$ contains no virtual elements;
  \item the number of terms in the $\fund$-expansion of $\key_{a}$ is the size of $\SYT(\sort(a))$;
  \item both $\key_a$ and $\key_{0^m \times a}$ have the same number of terms in their $\fund$-expansions for some $m>0$;
  \item both $\key_a$ and $\key_{0^m \times a}$ have the same number of terms in their $\fund$-expansions for any $m>0$.
  \end{enumerate}
  In particular, $\key_a$ is $\fund$-stable if and only if any one of these conditions is met.
\end{proposition}

\begin{definition}
  A weak dual equivalence for $(\mathcal{A},\des)$ is \emph{stable} if the restricted dual equivalence classes of degrees up to $6$ are $\fund$-stable key polynomials.
  \label{def:deg-weak-stable}
\end{definition}

While Definition~\ref{def:deg-weak-stable} looks like a local condition, the following result shows that it is global.

\begin{theorem}
  Let $\mathcal{A}$ be a set of combinatorial objects for which $\des$ is never $\varnothing$ (i.e. $\mathcal{A}$ has no virtual elements). If $\{\psi_i\}$ is a stable weak dual equivalence for $(\mathcal{A},\des)$, and $U \in \mathcal{A}$, then
  \begin{equation}
    \sum_{T \in [U]} \fund_{\des(T)} \ = \ \key_{a}
  \end{equation}
  for some key-stable weak composition $a$. In particular, the fundamental slide generating polynomial for $\mathcal{A}$ is key positive.
  \label{thm:positivity-key}
\end{theorem}

\begin{proof}
  We may assume $\mathcal{A}$ has a unique equivalence class under $\{\psi_i\}$. Since each $\key_a$ appearing as the generating polynomial for a restricted equivalence class is stable, there is a $\Des$-preserving bijection $\Key(a) \rightarrow \SYT(\sort(a))$. Therefore $\psi$ induces a dual equivalence on $(\mathcal{A},\Des)$. In particular, since there is a unique equivalence class, we have a $\Des$-preserving bijection $\Phi:\mathcal{A}\stackrel{\sim}{\rightarrow}\SYT(\lambda)$ for some partition $\lambda$. Given any weak composition $a$ for which $\sort(a)=\lambda$, by Theorem~\ref{thm:dual-SKT}, we may extend $\Phi$ to a $\Des$-preserving bijection $\Psi_{a}:\mathcal{A}\stackrel{\sim}{\rightarrow}\Key(a)$. We will show by induction on the size of $\lambda$ that there exists a (unique) weak composition $a$ with $\sort(a)=\lambda$ such that $\Psi_a$ is $\des$-preserving.

  If $\lambda$ is a single row, then $\mathcal{A}$ has a single element, say $A$, and we may take $a = \des(A)$. Assume, then, that $\lambda$ has at least two rows, as we proceed by induction on $|\lambda|$. For $|\lambda| \leq 6$, this follows immediately from the definition, thus establishing the base case. Assume $|\lambda| = n > 6$, and assume the result for strictly smaller partitions. 

  Let $m = \#\{\lambda_i\}$ be the number of removable corners of $\lambda$. Break $\mathcal{A}$ into equivalence classes under $\psi_3,\ldots,\psi_{n-1}$, say $\mathcal{A}^{(1)},\ldots,\mathcal{A}^{(m)}$. By induction, for each $i$ there is a bijection $\Psi^{(i)}:\mathcal{A}^{(i)}\stackrel{\sim}{\rightarrow}\Key(\hat{a}^{(i)})$ for some unique weak composition $\hat{a}^{(i)}$ of $n-1$ such that $\des_{(2,n)}(A) = \des(\Psi^{(i)}(A))$. For fixed $i$, let $m_i\geq 0$ be the length of the row of $\sort(\hat{a}^{(i)})$ to which a cell is added to obtain $\lambda$ (since we necessarily have $\sort(\hat{a}^{(i)}) \subset \lambda$). Let $k_i$ be the lowest part of $\hat{a}^{(i)}$ that equals $m_i$ if $m_i>0$, or take $k_i$ to be the largest index that occurs as the smallest nonzero entry among $\des(A)$ for $A\in\mathcal{A}^{(i)}$. Set $a^{(i)}$ to be the weak composition of $n$ obtained by adding $1$ to $\hat{a}^{(i)}_{k_i}$. Then we can lift $\Psi^{(i)}$ to a $\des$-preserving injection $\Psi^{(i)}:\mathcal{A}^{(i)}\stackrel{\sim}{\rightarrow}\Key(a^{(i)})$ where the image is those key tableaux with $1$ in row $k_i$. We claim that $a^{(i)} = a^{(j)} = a$, in which case the injections $\Psi^{(i)}$ combine to give the desired $\des$-preserving bijection $\Psi:\mathcal{A}\stackrel{\sim}{\rightarrow}\Key(a)$.

  Using the injections $\Psi^{(i)}$, we may identify each $A \in \mathcal{A}^{(i)}$ with an element of $\Key(a^{(i)})$. We recall some basic properties inherited from dual equivalence. For any $A \in \mathcal{A}$, we may use $\psi_5,\ldots,\psi_{n-1}$ to move $4,\ldots,n$ into any positions which they can occupy with $1,2,3$ in some fixed positions. Moreover, $\psi_j \psi_2 = \psi_2 \psi_j$ for any $j\geq 5$ and $\des_{(4,n)}(\psi_2(A)) = \des_{(4,n)}(A)$. Given $i,j$, there exists $B\in\mathcal{A}^{(i)}$ and $C\in\mathcal{A}^{(j)}$ such that $C = \psi_2(B)$. Combining this, we must a $\des_{(4,n)}$-preseving bijection between the class of $B$ generated by $\psi_5,\ldots,\psi_{n-1}$. In particular, the shapes of $a^{(i)}$ and $a^{(j)}$ must agree after deleting the fixed positions for $1,2,3$ from $B$ and $C$, respectively. Therefore the result follows from the fact that the weak compositions under $\psi_2,\psi_3,\psi_4$ must give a key polynomial.
\end{proof}

As demonstrated in the proof of Theorem~\ref{thm:positivity-key}, a stable weak dual equivalence $\{\psi_i\}$ for $(\mathcal{A},\des)$ induces a $\des$-preserving map $\Psi$ from $\mathcal{A}$ to $\Key$ such that, for any dual equivalence class $\mathcal{C}$ for $\mathcal{A}$ under $\{\psi_i\}$, the restriction of $\Psi$ to $\mathcal{C}$ gives a bijection with $\Key(a)$ for some (unique) weak composition $a$.

\begin{definition}
  Given a stable weak dual equivalence $\{\psi_i\}$ for $(\mathcal{A},\des)$, we call the induced map $\Psi:\mathcal{A}\rightarrow\Key$ the \emph{weak rectification map for $\mathcal{A}$ with respect to $\{\psi_i\}$}. For $A \in\mathcal{A}$, we say that $A$ \emph{weakly rectifies} to $\Psi(A)$.
  \label{def:rectification-weak}
\end{definition}
  
While the stability condition seems restrictive, provided the set $\mathcal{A}$ behaves nicely under stabilization, we can apply it more generally. As a first application of this new framework, we note that the involutions $\mathfrak{d}_i$ on reduced expressions defined in Definition~\ref{def:dual-stanley} give a weak dual equivalence. For example, the equivalence class in Figure~\ref{fig:weak-deg-red} has generating polynomial
\begin{displaymath}
  \fund_{(0,0,3,1,0,1)} + \fund_{(0,2,2,0,0,1)} + \fund_{(0,1,3,0,0,1)} + \fund_{(2,1,2,0,0,0)} + \fund_{(1,2,2,0,0,0)} + \fund_{(1,1,3,0,0,0)} = \key_{(0,0,3,1,0,1)} .
\end{displaymath}

\begin{figure}[ht]
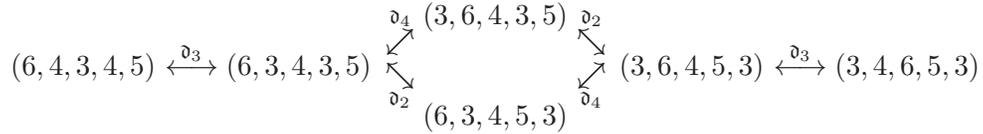

  \begin{displaymath}
    (6,4,3,4,5) \stackrel{\mathfrak{d}_3}{\longleftrightarrow} (6,3,4,3,5)
    \begin{array}{c}
      \stackrel{\mathfrak{d}_4}{\mathrel{\text{\ooalign{$\swarrow$\cr$\nearrow$}}}}
      \raisebox{2ex}{$(3,6,4,3,5)$}
      \stackrel{\mathfrak{d}_2}{\mathrel{\text{\ooalign{$\searrow$\cr$\nwarrow$}}}} \\
      \stackrel{\displaystyle\mathrel{\text{\ooalign{$\searrow$\cr$\nwarrow$}}}}{_{\mathfrak{d}_2}}
      \raisebox{-2ex}{$(6,3,4,5,3)$}
      \stackrel{\displaystyle\mathrel{\text{\ooalign{$\swarrow$\cr$\nearrow$}}}}{_{\mathfrak{d}_4}}
    \end{array}
    (3,6,4,5,3) \stackrel{\mathfrak{d}_3}{\longleftrightarrow} (3,4,6,5,3)
  \end{displaymath}
  \caption{\label{fig:weak-deg-red}A stable weak dual equivalence class for $R(1^2\times 42153)$.}
\end{figure}

\begin{lemma}
  Let $w$ be a permutation for which no element of $\R(w)$ is virtual. Then the maps $\mathfrak{d}_i$ on reduced expressions give a stable weak dual equivalence for $(\R(w),\des)$.
  \label{lem:red-weak}
\end{lemma}

\begin{proof}
  By Theorem~\ref{thm:deg-red}, the maps $\{\mathfrak{d}_i\}$ are involutions on $\R(w)$, and $\mathfrak{d}_i$ and $\mathfrak{d}_j$ commute for $|i-j|\geq 3$. Therefore we need only consider restricted dual equivalence classes under $\mathfrak{d}_h,\ldots,\mathfrak{d}_i$ for $i-h \leq 3$. We begin with the analysis of descent compositions in the proof of Theorem~\ref{thm:deg-red} and consider the consequences for considering weak descent compositions instead.

  For $i-h=0$, $\mathfrak{d}_i$ acts trivially if and only if $\rho_{i-1}\rho_{i}\rho_{i+1}$ is weakly increasing or weakly decreasing. Since, by the reduced condition, consecutive letters may not be equal, the sequence must be strict, and so the descent composition is $(3)$ or $(1,1,1)$. Since $\key_a = \fund_a$ for $a$ any weak composition that flattens to $(n)$ or $(1^n)$, in either case the class is a single key polynomial. 

  If $\mathfrak{d}_i$ acts nontrivially, then it pairs $acb$ with $cab$ (if $\mathfrak{d}_i = \swap_{i-1}$) or $bac$ with $bca$ (if $\mathfrak{d}_i = \swap_{i}$) or $aba$ with $bab$ (if $\mathfrak{d}_i = \braid_{i}$), where $a<b<c$. Taking the first case, the run decompositions are $(ac|b)$ and $(c|ab)$, giving weak descent compositions are $(0^{a-2},1,2)$ and $(0^{a-1},2,0^{c-a-1},1)$, and this gives $\key_{(0^{a-1},2,0^{c-a-1},1)}$. Similarly, in the second case we obtain $\fund_{(0^{a-1},2,0^{b-a-1},1)} + \fund_{(0^{a-1},1,0^{b-a-1},2)} = \key_{(0^{a-1},1,0^{b-a-1},2)}$. Finally, in the third case, we have $\fund_{(0^{a-2},1,2)} + \fund_{(0^{a-1},2,1)} = \key_{(0^{a-1},2,1)}$.

  As before, a similar by hand analysis can handle the cases $i-h=1,2,3$, though it is more efficient (and certainly less error-prone) to program the involutions on a computer and check them for all reduced words on $\{1,\ldots,i-h+3\}$.
\end{proof}

In particular, we have a simple, combinatorial proof of the key positivity of Schubert polynomials. Note that, as we show below, this holds for \emph{any} permutation $w$, not only for those with no virtual reduced expressions. For example, for reduced expressions for $1^2 \times 42153$, the reduced expressions in Figure~\ref{fig:weak-deg-red} rectify to the standard key tableaux in Figure~\ref{fig:key-deg-rect}, respectively, after stabilizing by prepending $0^2$. We may equally well use the induced bijection between the corresponding reduced expressions for $42153$ and $\Key(3,1,0,1)$, where the virtual terms coincide.

\begin{figure}[ht]
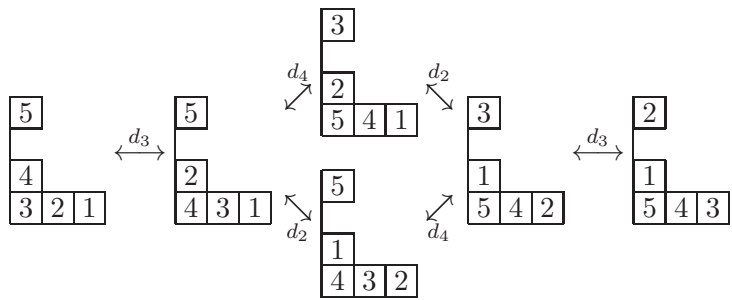

  \begin{displaymath}
    \raisebox{1\cellsize}{$\vline\tableau{5 \\ \\ 4 \\ 3 & 2 & 1}$}
    \stackrel{d_3}{\longleftrightarrow}
    \raisebox{1\cellsize}{$\vline\tableau{5 \\ \\ 2 \\ 4 & 3 & 1}$}
    \begin{array}{c}
      \stackrel{d_4}{\mathrel{\text{\ooalign{$\swarrow$\cr$\nearrow$}}}}
      \raisebox{2\cellsize}{$\vline\tableau{3 \\ \\ 2 \\ 5 & 4 & 1}$}
      \stackrel{d_2}{\mathrel{\text{\ooalign{$\searrow$\cr$\nwarrow$}}}} \\ \\
      \stackrel{\displaystyle\mathrel{\text{\ooalign{$\searrow$\cr$\nwarrow$}}}}{_{d_2}}
      \raisebox{1\cellsize}{$\vline\tableau{5 \\ \\ 1 \\ 4 & 3 & 2}$}
      \stackrel{\displaystyle\mathrel{\text{\ooalign{$\swarrow$\cr$\nearrow$}}}}{_{d_4}}
    \end{array}
    \raisebox{1\cellsize}{$\vline\tableau{3 \\ \\ 1 \\ 5 & 4 & 2}$}
    \stackrel{d_3}{\longleftrightarrow}
    \raisebox{1\cellsize}{$\vline\tableau{2 \\ \\ 1 \\ 5 & 4 & 3}$}
  \end{displaymath}
  \caption{\label{fig:key-deg-rect}The weak dual equivalence class for $\Key(3,1,0,1)$.}
\end{figure}

\begin{definition}
  Given a permutation $w$, say that a reduced expression $\rho\in\R(w)$ is \emph{yamanouchi} if it weakly rectifies to a yamanouchi key tableau.
  \label{def:yam-red}
\end{definition}
  
For example, the yamanouchi reduced expressions for $42153$ are shown in Figure~\ref{fig:yam-red}.

\begin{figure}[ht]
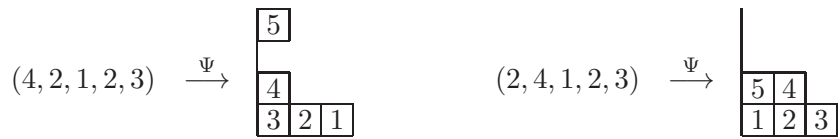

  \begin{displaymath}
    \begin{array}{cccc@{\hskip 4em}ccc}
      (4,2,1,2,3) & \stackrel{\Psi}{\longrightarrow} & \raisebox{1.5\cellsize}{$\vline\tableau{5 \\ \\ 4 \\ 3 & 2 & 1}$} & &
      (2,4,1,2,3) & \stackrel{\Psi}{\longrightarrow} & \raisebox{1.5\cellsize}{$\vline\tableau{\\ \\ 5 & 4 \\ 1 & 2 & 3}$}
    \end{array}
  \end{displaymath}
  \caption{\label{fig:yam-red}The yamanouchi elements of $\R(42153)$ and their weak rectifications.}
\end{figure}

\begin{theorem}
  For $w$ a permutation, we have
  \begin{equation}
    \schubert_w = \sum_{\substack{\rho\in\R(w) \\ \rho \ \mathrm{yamanouchi}}} \key_{\des(\rho)} = \sum_{a} c_{w,a} \key_{a},
  \end{equation}
  where $c_{w,a}$ is the number of yamanouchi reduced expressions for $w$ with weak descent composition $a$. In particular, the Schubert polynomial $\schubert_w$ is key positive.
  \label{thm:sch-key}
\end{theorem}

\begin{proof}
  For $\schubert_w$ an $\fund$-stable polynomial, the result follows from Lemma~\ref{lem:red-weak} and Theorem~\ref{thm:positivity-key}. For $w$ not $\fund$-stable, by \cite{AS17} there exists a (specific) nonnegative integer $\eta(w)$ for which $1^{\eta(w)}\times w$ is $\fund$-stable. Let $\Psi$ be the induced $\des$-preserving map from $\R(1^{\eta(w)}\times w)$ to $\Key$. For any nonvirtual elements of $R(w)$, the corresponding reduced expressions for $\R(1^{\eta(w)}\times w)$ are precisely those that map to some standard key tableau with weak descent composition $0^{\eta(w)}\times a$. Given any weak dual equivalence class, we may pull back both the reduced expressions in $\R(1^{\eta(w)}\times w)$ that are nonvirtual in $\R(w)$ and those standard key tableaux that have at least $\eta(w)$ leading $0$'s. This gives a $\des$-preserving bijection between nonvirtual elements, so they must have the same generating polynomial. Hence $\schubert_w$ is also key positive with the same leading terms.
\end{proof}

%
\section{Littlewood--Richardson rules}
%
\label{sec:LRR}

\subsection{Shuffle products and Schur products}
\label{sec:LRR-schur}

Gessel used the shuffle product of Eilenberg and Mac Lane \cite{EM53} to give a Littlewood--Richardson rule for fundamental quasisymmetric functions \cite{Ges84}. The \emph{shuffle product} of words $A$ and $B$, denoted by $A \shuffle B$, is the set of all ways of riffle shuffling the terms of $A$, in order, with the terms of $B$, in order. 

\begin{definition}
  The \emph{shuffle product} of strong compositions $\alpha$ and $\beta$ is the formal series
  \begin{equation}
    \alpha\shuffle\beta = \sum_{C \in A \shuffle B} \Des(C),
  \end{equation}
  where $A,B$ are words in disjoint alphabets with $\Des(A)=\alpha$ and $\Des(B)=\beta$.
\label{def:shuffle}
\end{definition}

For example, to compute $(2,3) \shuffle (2,1)$, we may take descent compositions for all $\binom{8}{5}$ shuffles of $22111 \shuffle 443$, though we could equally well take shuffles of $33133 \shuffle 682$.

\begin{theorem}[\cite{Ges84}]
  For strong compositions $\alpha,\beta$, we have
  \begin{equation}
    F_{\alpha} F_{\beta} = \sum_{\gamma} [\gamma \mid \alpha \shuffle \beta] F_{\gamma},
    \label{e:shuffle}
  \end{equation}
  where $[\gamma \mid \alpha \shuffle \beta]$ denotes the multiplicity of $\gamma$ in the shuffle product $\alpha \shuffle \beta$.
  \label{thm:shuffle}
\end{theorem}

For example, $F_{(3,2,3)}$ appears with multiplicity $2$ in the product $F_{(2,3)} F_{(2,1)}$, corresponding to the shuffles $22414113$ and $24423111$ of $22111 \shuffle 443$.

Given partitions $\mu,\nu$, define the diagram $\mu\otimes\nu$ to be the concatentation of Young diagrams for $\mu$ and $\nu$. A \emph{standard Young tableau of shape $\mu\otimes \nu$} is a bijective filling of $\mu\otimes\nu$ with entries $\{1,2,\ldots,n\}$ such that each shape satisfies the Young tableaux conditions. Extend descent compositions to product shapes using the column words with entries of $\mu$ read before entries of $\nu$. For example, the standard Young tableau of shape $(3,2)\otimes(2,1)$ on the right side of Figure~\ref{fig:shuffle} has descent composition $(3,2,3)$.

\begin{figure}[ht]
  \begin{displaymath}
    \begin{array}{ccccccc}
      \tableau{3 & 4 \\ 1 & 2 & 5} & \times & \tableau{3 \\ 1 & 2} & \times & AABABAAB &
      \longleftrightarrow & \tableau{4 & 6 \\ 1 & 2 & 7} \otimes \tableau{8 \\ 3 & 5}
    \end{array}
  \end{displaymath}
  \caption{\label{fig:shuffle}An example of the bijection from $\SYT(3,2) \times \SYT(2,1) \times A^{5} \shuffle B^{3}$ to $\SYT((3,2)\otimes(2,1))$.}
\end{figure}

The following expansion is an immediate consequence of Theorem~\ref{thm:shuffle}.

\begin{proposition}
  For partitions $\mu,\nu$, we have
  \begin{equation}
    s_{\mu} s_{\nu} = \sum_{T \in \SYT(\mu\otimes\nu)} F_{\Des(T)}.
  \end{equation}
\end{proposition}

  
We may extend the dual equivalence operators $d_i$ to products of tableaux using the column reading word to obtain the classical Littlewood-Richardson rule for Schur functions.

\begin{theorem}[\cite{Ass15}]
  The maps $\{d_i\}$ give a dual equivalence for $(\SYT(\mu\otimes\nu),\Des)$. In particular, we have
  \begin{equation}
    s_{\mu} s_{\nu} = \sum_{\lambda} c_{\mu,\nu}^{\lambda} s_{\lambda},
    \label{e:LRR-schur}
  \end{equation}
  where $c_{\mu,\nu}^{\lambda}$ is the number of super-standard tableaux of shape $\mu\otimes\nu$ that rectify to $\lambda$.
  \label{thm:LRR-schur}
\end{theorem}

\begin{remark}
  Dual equivalence can rediscover the \emph{jeu de taquin} algorithm of Sch{\"u}tzenberger \cite{Sch77} that gives an explicit rectification process. Indeed, Haiman originally called his involutions dual equivalence since they are precisely dual to \emph{jeu de taquin}. That is, dual equivalence moves commute with \emph{jeu de taquin} moves and as such can be used to give a simple proof that \emph{jeu de taquin} is well-defined and provides the explicit map from skew shapes to straight shapes needed to prove Theorem~\ref{thm:LRR-schur}.
  \label{rmk:jdt}
\end{remark}

Adjoint to products of shapes, whenever $\mu \subseteq \lambda$ we may define the \emph{skew shape} $\lambda/\mu$ to be the set-theoretic difference between the two shapes. Define the corresponding \emph{skew Schur function} by
\begin{equation}
  s_{\lambda/\mu} = \sum_{T \in \SYT(\lambda/\mu)} F_{\Des(\lambda)}.
\end{equation}

Once again, the dual equivalence operators on tableaux apply, giving the following.

\begin{theorem}[\cite{Ass15}]
  The maps $\{d_i\}$ give a dual equivalence for $(\SYT(\lambda/\mu),\Des)$. In particular,
  \begin{equation}
    s_{\lambda/\mu} = \sum_{\nu} c_{\mu,\nu}^{\lambda} s_{\nu},
    \label{e:skew-schur}
  \end{equation}
  where $c_{\mu,\nu}^{\lambda}$ is the number of super-standard tableaux of shape $\lambda/\mu$ that rectify to $\nu$.
  \label{thm:skew-schur}
\end{theorem}

\begin{figure}[ht]
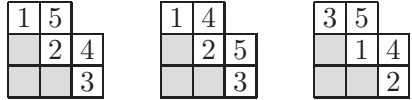

  \begin{displaymath}
    \begin{array}{c@{\hskip 2em}c@{\hskip 2em}c@{\hskip 2em}c@{\hskip 2em}c}
      \tableau{1 & 5 \\ \graybox & 2 & 4 \\ \graybox & \graybox & 3 } &
      \tableau{1 & 4 \\ \graybox & 2 & 5 \\ \graybox & \graybox & 3 } &
      \tableau{3 & 5 \\ \graybox & 1 & 4 \\ \graybox & \graybox & 2 } 
    \end{array}
  \end{displaymath}
  \caption{\label{fig:SYT-skew}The super-standard skew tableaux of shape $(3,3,2)/(2,1)$.}
\end{figure}

For example, from Figure~\ref{fig:SYT-skew}, we compute
\begin{displaymath}
  s_{(3,3,2)/(2,1)}(X) = s_{(3,1,1)}(X) + s_{(3,2)}(X) + s_{(2,2,1)}(X).
\end{displaymath}

Since skewing is adjoint to multiplication, the repetition of notation in Theorems~\ref{thm:LRR-schur} and \ref{thm:skew-schur} is intentional. That is, $c_{\mu,\nu}^{\lambda}$ is well-defined by either and agrees for both.

\subsection{Skew key polynomials}
\label{sec:LRR-skew}

Given weak compositions $a,d$, we say that $a \subseteq d$ if $a_i \leq d_i$ for all $i$. Equivalently, the key diagram for $a$ is a subset of the key diagram for $d$. This allows us to form the \emph{skew key diagram} $d/a$ as the set-theoretic difference between the two key diagrams. We extend the notion of standard key tableaux to skew shapes, but now we allow a new type of column inversion.

\begin{definition}
  For weak compositions $a \subseteq d$, a \emph{standard skew key tableau of shape $d/a$} is a bijective filling of the skew key diagram $d/a$ with entries $\{1,2,\ldots,n\}$ such that rows weakly decrease and if some entry $i$ is above and in the same column as an entry $k$ with $i<k$, then either there is an entry immediately right of $k$, say $j$, and $i<j$ or there is a skewed cell immediately left of $k$ and an entry $j<k$ immediately left of $i$. 
  \label{def:skew-SKT}
\end{definition}

Denote the set of skew standard key tableau of shape $d/a$ by $\Key(d/a)$. As with the case of products, we use Definition~\ref{def:des-std} to define weak descent compositions for standard skew key tableaux. For example, see Figure~\ref{fig:key-skew-2}. 


\begin{figure}[ht]
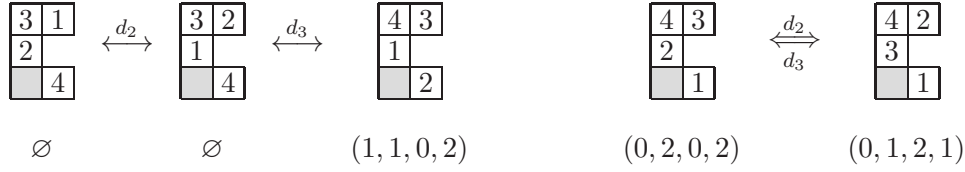

  \begin{displaymath}
    \begin{array}{ccccc@{\hskip 4\cellsize}cccc}
      \vline\tableau{3 & 1 \\ 2 \\ \graybox & 4 \\ } &
      \raisebox{-0.5\cellsize}{$\stackrel{d_2}{\longleftrightarrow}$}  &
      \vline\tableau{3 & 2 \\ 1 \\ \graybox & 4 \\ } &
      \raisebox{-0.5\cellsize}{$\stackrel{d_3}{\longleftrightarrow}$}  &
      \vline\tableau{4 & 3 \\ 1 \\ \graybox & 2 \\ } & &
      \vline\tableau{4 & 3 \\ 2 \\ \graybox & 1 \\ } &
      \raisebox{-\cellsize}{$\stackrel{\displaystyle\stackrel{d_2}{\Longleftrightarrow}}{\scriptstyle d_3}$}  &
      \vline\tableau{4 & 2 \\ 3 \\ \graybox & 1 \\ } \\ \\
      \varnothing & &
      \varnothing & &
      (1,1,0,2) & &
      (0,2,0,2) & &
      (0,1,2,1) 
    \end{array}
  \end{displaymath}
  \caption{\label{fig:key-skew-2}The skew key tableaux of shape $(0,2,1,2)/(0,1,0,0)$.}
\end{figure}

\begin{definition}
  For weak compositions $a \subseteq d$, the \emph{skew key polynomial of shape $d/a$} is
  \begin{equation}
    \key_{d/a} = \sum_{T \in \Key(d/a)} \fund_{\des(T)}.
  \end{equation}
  \label{def:skew-key}
\end{definition}

Observe that $\key_{d/a}$ is not necessarily nonnegative in the key basis. For example,
\begin{displaymath}
  \key_{(0,2,1,2)/(0,1,0,0)} = \key_{(0, 2, 0, 2)}+\key_{(1, 1, 0, 2)}+\key_{(0, 1, 2, 1)}-\key_{(0, 2, 1, 1)}-\key_{(1, 2, 0, 1)}-\key_{(1, 1, 2, 0)}+\key_{(1, 2, 1, 0)}.
\end{displaymath}
Furthermore, it is not the case that the skew key polynomial for $d/a$ stabilizes to a skew Schur function for $\sort(d)/\sort(a)$; for example $\key_{(3,2,3)/(0,1,2)}$ stabilizes to $s_{(3,3,2)/(2,1)} - s_{(2,2,1)}$. We do, however, get an \emph{injective} map from $\Key(d/a)$ to $\SYT(\sort(d)/\sort(a))$. To prove this, we extend the dual equivalence operators $d_i$ from Definition~\ref{def:key-deg} to skew key tableaux.

\begin{theorem}
  The maps $\{d_i\}$ give a dual equivalence for $(\Key(d/a),\Des)$. In particular,
  \begin{equation}
    \lim_{m\rightarrow\infty} \key_{(0^m\times d)/(0^m\times a)} = \sum_{\nu} c_{a,\nu}^{d} s_{\nu},
    \label{e:skew-key-limit}
  \end{equation}
  where $c_{\mu,\nu}^{\lambda}$ is the number of super-standard skew tableaux of shape $d/a$ that rectify to $\nu$.
  \label{thm:skew-key-limit}
\end{theorem}

\begin{proof}
  It does not follow from Lemma~\ref{lem:dual-SKT} that $d_i$ is well-defined on skew diagrams. Certainly if $d_i$ acts by swapping $i$ and $i\pm 1$, the two entries are not in the same row or column, so the map is well-defined in this case. However, for $\braid_i$, we must be more careful since now we can have two of $i,i\pm 1$ in the same row with $i\mp 1$ in a lower row if it lies in the right column with a skewed cell to its left (for example, see the second and third tableaux from the left in Figure~\ref{fig:key-skew-2}). This case is effectively a rotation of the braid case for straight diagrams, so this action is also well-defined.

  Consider the bijection $\Phi:\Key(a)\stackrel{\sim}{\rightarrow}\Key(\sort(a))$ from the proof of Theorem~\ref{thm:dual-SKT} defined by letting the cells of $T$ fall to shape $\sort(a)$, then sorting the columns to descend upward, and replacing $i$ with $n-i+1$. If we have skew cells regarded as smaller than any others, then this maps extends to the skew case and still commutes with the maps $d_i$ in the sense that $\Phi(d_i(T)) = d_{n-i+1}(\Phi(T))$, hence it is a bijection $\Phi:\Key(d/a)\stackrel{\sim}{\rightarrow}\Key(\sort(d)/\sort(a))$ and the result follows.
\end{proof}

As with the case of products, these maps do not, in general, give a weak dual equivalence as evidenced by Figure~\ref{fig:key-skew-2}. However, as with products, they do in the case when $a$ is nondecreasing.

\begin{theorem}
    For $\lambda$ a partition of length $n$, let $a_{\lambda}$ be the weakly increasing weak composition of length $n$ that sorts to $\lambda$. For $d$ any weak composition of length $n$ with $a \subseteq d$, the maps $\{d_i\}$ give a weak dual equivalence for $(\Key(d/a_{\lambda}),\des)$. In particular,
  \begin{equation}
    \key_{d/a_{\lambda}} = \sum_{b} \hat{c}_{a_{\lambda},b}^{d} \key_{b},
    \label{e:skew-key}
  \end{equation}
  where $\hat{c}_{a_{\lambda},b}^{d}$ is the number yamanouchi skew key tableaux of shape $d/a_{\lambda}$ that rectify to $b$.
  \label{thm:skew-key}
\end{theorem}

\begin{proof}
  When $a_{\lambda}$ is a partition shape that sits at the top row of $d$, then if $i$ lies strictly below and strictly right of $i+1$, then there must be a cell (not skewed) containing an entry, say $k$, below $i+1$ in the same column and left of $i$ in the same row, and necessarily $k>i+1$ so $k$ lies in a different block of the run decomposition than $i+1$ and $i$. Therefore in constructing the weak descent composition, the run block containing $k$ will be indexed by some row weakly below that of $i$, and so the run block for $i+1$ will be indexed by some row strictly below $i$. Thus we may equivalently define the weak descent compositions by taking the largest entry, instead of the lowest entry, and now the definition coincides with that for a non-skewed key tableaux. Therefore the result follows from Theorem~\ref{thm:deg-key}.
\end{proof}

\begin{figure}[ht]
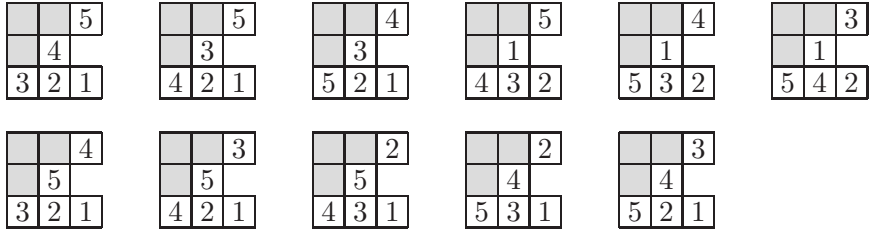

  \begin{displaymath}
    \begin{array}{c@{\hskip 2em}c@{\hskip 2em}c@{\hskip 2em}c@{\hskip 2em}c@{\hskip 2em}c}
      \tableau{ \graybox & \graybox & 5 \\ \graybox & 4 \\ 3 & 2 & 1} &
      \tableau{ \graybox & \graybox & 5 \\ \graybox & 3 \\ 4 & 2 & 1} &
      \tableau{ \graybox & \graybox & 4 \\ \graybox & 3 \\ 5 & 2 & 1} &
      \tableau{ \graybox & \graybox & 5 \\ \graybox & 1 \\ 4 & 3 & 2} &
      \tableau{ \graybox & \graybox & 4 \\ \graybox & 1 \\ 5 & 3 & 2} &
      \tableau{ \graybox & \graybox & 3 \\ \graybox & 1 \\ 5 & 4 & 2} \\ \\
      \tableau{ \graybox & \graybox & 4 \\ \graybox & 5 \\ 3 & 2 & 1} &
      \tableau{ \graybox & \graybox & 3 \\ \graybox & 5 \\ 4 & 2 & 1} &
      \tableau{ \graybox & \graybox & 2 \\ \graybox & 5 \\ 4 & 3 & 1} &
      \tableau{ \graybox & \graybox & 2 \\ \graybox & 4 \\ 5 & 3 & 1} &
      \tableau{ \graybox & \graybox & 3 \\ \graybox & 4 \\ 5 & 2 & 1} &
    \end{array}
  \end{displaymath}
  \caption{\label{fig:key-skew}The skew key tableaux of shape $(3,2,3)/(0,1,2)$.}
\end{figure}

For example, each row of standard skew key tableaux in Figure~\ref{fig:key-skew} is a dual equivalence class, and from the weak descent compositions we compute
\begin{displaymath}
  \key_{(3,2,3)/(0,1,2)} = \key_{(3,1,1)} + \key_{(3,2,0)}.
\end{displaymath}
In contrast, for the corresponding skew Schur function we have
\begin{displaymath}
  s_{(3,3,2)/(2,1)} = s_{(3,1,1)} + s_{(3,2)} + s_{(2,2,1)}.
\end{displaymath}
Note that for $b$ any weak composition with $\sort(b)=(2,2,1)$ and any $m\geq 0$, the term $\key_{0^m\times(3,2,3)}$ does not appear in the key expansion of the product $\key_b \key_{0^{m}\times (0,1,2)}$. 

\begin{remark}
  We can describe the rectification rule more directly by first applying the injection $\Key(d/a) \rightarrow \SYT(\sort(d)/\sort(a))$, then rectifying according to \emph{jeu de taquin}, then applying the bijection $\SYT(\mu) \stackrel{\sim}{\rightarrow} \Key(c)$. In so doing, one sees that the rectification map can be described more directly by a process similar to \emph{jeu de taquin}, though now entries may cycle when a cell slides left, so the explicit algorithm becomes more involved.
  \label{rmk:jdt-key}
\end{remark}

\subsection{Products of key polynomials}
\label{sec:LRR-key}

Assaf and Searles generalized the shuffle product to weak compositions to give a Littlewood--Richardson rule for fundamental slide polynomials \cite{AS17}. 


\begin{definition}
  The \emph{slide product} of weak compositions $a$ and $b$ of length $n$ is the formal series
  \begin{equation}
    a \shuffle b = \sum_{C \in A \shuffle B} \des(C),
  \end{equation}
  where $A = (2n-1)^{a_1} \cdots (3)^{a_{n-1}} (1)^{a_n}$ and $B = (2n)^{b_1} \cdots (4)^{b_{n-1}} (2)^{b_n}$, and for $C\in A \shuffle B$ with run decomposition $(C^{(k)} | \cdots | C^{(1)})$, we set $c_i = \min(\lceil C^{(i)}_1/2 \rceil,c_{i+1}-1)$ and define $\des(C)$ by $\des(C)_{c_i} = |C^{(i)}|$ and all other parts are zero if all $c_i>0$ and $\des(C) = \varnothing$ otherwise.
\label{def:slide}
\end{definition}

For example, to compute $(2,0,3) \shuffle (0,2,1)$, we may take weak descent compositions for the non-virtual shuffles of $55111 \shuffle 664$. For instance, $\des(56645111) = (3,2,3)$ whereas $\des(55616114) = \varnothing$ since $r_1=0$. Note that bumping the latter example, we have $\des(77838336) = (3,2,0,3)$.

\begin{theorem}[\cite{AS17}]
  For weak compositions $a,b$, we have
  \begin{equation}
    \fund_a \fund_b = \sum_{c} [c \mid a \shuffle b] \fund_c,
    \label{e:slide}
  \end{equation}
  where $[c \mid a \shuffle b]$ denotes the multiplicity of $c$ in the slide product $a \shuffle b$.
  \label{thm:slide}
\end{theorem}

For example, we compute the following slide product using only non-virtual terms
\begin{eqnarray*}
  \fund_{(2,0,3)} \fund_{(0,2,1)} & = & \fund_{(2,2,4)} + \fund_{(2,3,3)} + \fund_{(2,4,2)} + \fund_{(2,5,1)} + \fund_{(3,1,4)} + \fund_{(3,2,3)}\\
  &   & + \fund_{(3,3,2)} + \fund_{(3,4,1)} + \fund_{(4,0,4)} + \fund_{(4,1,3)} + \fund_{(4,2,2)} + \fund_{(4,3,1)}.
\end{eqnarray*}

Given weak compositions $a,b$, define the diagram $\YD(a\otimes b)$ to be the concatenation of key diagrams for $a$ and for $b$. A \emph{standard key tableau of shape $a\otimes b$} is a bijective filling of $\YD(a\otimes b)$ with entries $\{1,2,\ldots,n\}$ such that each diagram satisfies the key tableaux conditions. We now use the full power of Definition~\ref{def:des-std} to define weak descent composition for a product shapes. For example, the standard key tableau of shape $(0,2,0,3)\otimes(0,0,2,1)$ on the right side of Figure~\ref{fig:slide} has run decomposition $(876|54|321)$ and weak descent composition $(3,2,3,0)$.

\begin{figure}[ht]
  \begin{displaymath}
    \begin{array}{ccccccc}
      \vline\tableau{5 & 4 & 1 \\ \\ 3 & 2\\\hline} & \times & \vline\tableau{3 \\ 5 & 4 \\ \\\hline} & \times & AABABAAB
      &  \longleftrightarrow &
      \vline\tableau{8 & 7 & 2 \\ \\ 5 & 3 \\\hline} \raisebox{-.5\cellsize}{$\otimes$} \vline\tableau{1 \\ 6 & 4 \\ \\\hline}
    \end{array}
  \end{displaymath}
  \caption{\label{fig:slide}An example of the bijection from $\Key(2,0,3) \times \Key(0,2,1) \times A^{5} \shuffle B^{3}$ to $\Key((2,0,3)\otimes(0,2,1))$.}
\end{figure}

\begin{theorem}
  For weak compositions $a,b$, we have
  \begin{equation}
    \key_{a} \key_{b} = \sum_{T \in \Key(\mu\otimes\nu)} \fund_{\des(T)}.
  \end{equation}
  \label{thm:key-fund}
\end{theorem}

\begin{proof}
  Given $(T,U)\in\Key(a) \times \Key(b)$ and $C\in A^{|a|} \shuffle B^{|b|}$, we may construct an element of $\Key(a\otimes b)$ by placing the indices $\{i \mid C_i=A\}$ into $\YD(a)$ according to the order specified by $T$, and placing the indices $\{i \mid C_i=B\}$ into $\YD(b)$ according to the order specified by $U$. This process is clearly reversible, establishing a bijection
  \begin{displaymath}
    \Key(a) \times \Key(b)\times \left(A^{|a|} \shuffle B^{|b|}\right)
    \stackrel{\sim}{\longrightarrow} \Key(a\otimes b).
  \end{displaymath}
  Using this bijection and applying Theorem~\ref{thm:slide} gives
  \begin{displaymath}
    \key_{a} \key_{b} = \sum_{(T,U)\in\Key(a) \times \Key(b)} \fund_{\des(T)} \fund_{\des(U)} = \sum_{T\in\Key(a\otimes b)} \fund_{\des(T)}.
  \end{displaymath}
\end{proof}


Extend dual equivalence operators $d_i$ from Definition~\ref{def:key-deg} to standard key tableaux on products by asserting in Definition~\ref{def:key-deg} that $d_i(T)=T$ whenever $c=i$ (this is implicit in the original definition as shown in the proof of Lemma~\ref{lem:dual-SKT}). For example, Figure~\ref{fig:dual-key-product} gives one dual equivalence class for $\Key(0,2,1,0) \times \Key(0,1,0,1)$. Note that the generating function is $s_{(3,2)}$ and the generating polynomial is $\key_{(0,3,2,0)}$.

\begin{figure}[ht]
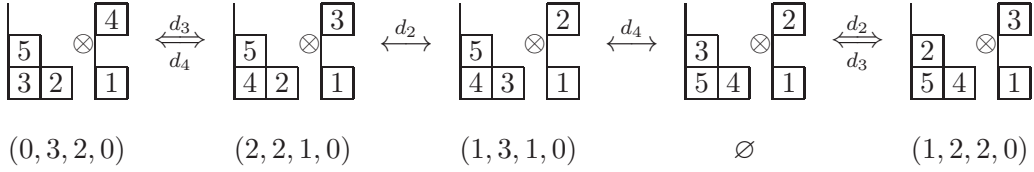

  \begin{displaymath}
    \begin{array}{ccccccccc}
      \vline\tableau{ \\ 5 \\ 3 & 2 \\\hline} \raisebox{-.5\cellsize}{$\otimes$} \vline\tableau{ 4 \\ \\ 1 \\\hline }
      & \raisebox{-\cellsize}{$\stackrel{\displaystyle\stackrel{d_3}{\Longleftrightarrow}}{\scriptstyle d_4}$}  &
      \vline\tableau{ \\ 5 \\ 4 & 2 \\\hline} \raisebox{-.5\cellsize}{$\otimes$} \vline\tableau{ 3 \\ \\ 1 \\\hline }
      & \raisebox{-0.5\cellsize}{$\stackrel{d_2}{\longleftrightarrow}$} &                                                
      \vline\tableau{ \\ 5 \\ 4 & 3 \\\hline} \raisebox{-.5\cellsize}{$\otimes$} \vline\tableau{ 2 \\ \\ 1 \\\hline } 
      & \raisebox{-0.5\cellsize}{$\stackrel{d_4}{\longleftrightarrow}$} &                                                
      \vline\tableau{ \\ 3 \\ 5 & 4 \\\hline} \raisebox{-.5\cellsize}{$\otimes$} \vline\tableau{ 2 \\ \\ 1 \\\hline }
      & \raisebox{-\cellsize}{$\stackrel{\displaystyle\stackrel{d_2}{\Longleftrightarrow}}{\scriptstyle d_3}$}  &
      \vline\tableau{ \\ 2 \\ 5 & 4 \\\hline} \raisebox{-.5\cellsize}{$\otimes$} \vline\tableau{ 3 \\ \\ 1 \\\hline } \\ \\
      (0,3,2,0) & & (2,2,1,0) & & (1,3,1,0) & & \varnothing & & (1,2,2,0)
    \end{array}
  \end{displaymath}
  \caption{\label{fig:dual-key-product}A dual equivalence class for $\Key((0,2,1,0) \otimes (0,1,0,1))$ along with their weak descent compositions.}
\end{figure}

\begin{proposition}
  The maps $\{d_i\}$ give a dual equivalence for $(\Key(a\otimes b),\Des)$. In particular,
  \begin{equation}
    \lim_{m\rightarrow\infty} \key_{0^m\times a} \key_{0^m\times b} = \sum_{\lambda} c_{\sort(a),\sort(b)}^{\lambda} s_{\lambda},
    \label{e:LRR-key-limit}
  \end{equation}
  where $c_{\mu,\nu}^{\lambda}$ is the number of super-standard tableaux of shape $\mu\otimes\nu$ that rectify to $\lambda$.
  \label{prop:LRR-key-limit}
\end{proposition}

\begin{proof}
  Borrowing from the proof of Theorem~\ref{thm:dual-SKT}, the bijection $\Phi:\Key(a)\rightarrow\SYT(\sort(a))$ easily extends to products while maintaining descents (with the exchange $i \mapsto n-i$) and commuting with dual equivalence involutions (with the indexing exchange $i \mapsto n-i+1$). Therefore the result follows from Theorem~\ref{thm:LRR-schur}.
\end{proof}

These involutions do not, in general, give a weak dual equivalence. For example, Figure~\ref{fig:bad-key-product} gives another dual equivalence class for $\Key(0,2,1,0) \otimes \Key(0,1,0,1)$. Note that the generating function is $s_{(2,2,1)}$ but the generating polynomial is
\[\key_{(2,2,0,1)} + \key_{(1,2,1,1)} + \key_{(1,2,2,0)} - \key_{(2,2,1,0)} - \key_{(1,2,1,1)},\]
which is not positive. Computing the full key product, some cancellation occurs, and we have
\begin{displaymath}
  \key_{(0,2,1,0)} \key_{(0,1,0,1)} = \key_{(0,3,1,1)} + \key_{(1,2,1,1)} + \key_{(0,3,2,0)} + \key_{(2,2,0,1)} + \key_{(1,2,2,0)} - \key_{(2,2,1,0)}.
\end{displaymath}

\begin{figure}[ht]
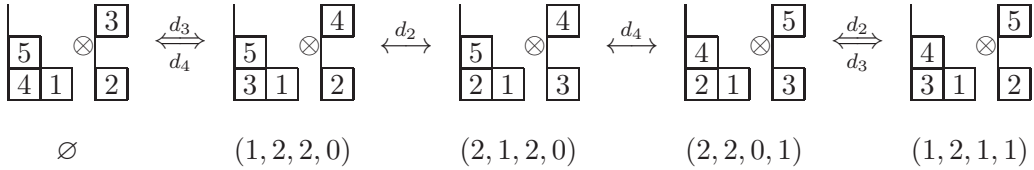

  \begin{displaymath}
    \begin{array}{ccccccccc}
      \vline\tableau{\\ 5 \\ 4 & 1 \\\hline} \raisebox{-.5\cellsize}{$\otimes$} \vline\tableau{ 3 \\ \\ 2 \\\hline }
      & \raisebox{-\cellsize}{$\stackrel{\displaystyle\stackrel{d_3}{\Longleftrightarrow}}{\scriptstyle d_4}$}  &
      \vline\tableau{\\ 5 \\ 3 & 1 \\\hline} \raisebox{-.5\cellsize}{$\otimes$} \vline\tableau{ 4 \\ \\ 2 \\\hline }
      & \raisebox{-0.5\cellsize}{$\stackrel{d_2}{\longleftrightarrow}$} &                                                
      \vline\tableau{\\ 5 \\ 2 & 1 \\\hline} \raisebox{-.5\cellsize}{$\otimes$} \vline\tableau{ 4 \\ \\ 3 \\\hline } 
      & \raisebox{-0.5\cellsize}{$\stackrel{d_4}{\longleftrightarrow}$} &                                                
      \vline\tableau{\\ 4 \\ 2 & 1 \\\hline} \raisebox{-.5\cellsize}{$\otimes$} \vline\tableau{ 5 \\ \\ 3 \\\hline }
      & \raisebox{-\cellsize}{$\stackrel{\displaystyle\stackrel{d_2}{\Longleftrightarrow}}{\scriptstyle d_3}$}  &
      \vline\tableau{\\ 4 \\ 3 & 1 \\\hline} \raisebox{-.5\cellsize}{$\otimes$} \vline\tableau{ 5 \\ \\ 2 \\\hline } \\ \\
      \varnothing & & (1,2,2,0) & & (2,1,2,0) & & (2,2,0,1) & & (1,2,1,1)
    \end{array}
  \end{displaymath}
  \caption{\label{fig:bad-key-product}A dual equivalence class for $\Key(0,2,1,0) \times \Key(0,1,0,1)$ along with their weak descent compositions.}
\end{figure}

The problem that arises in this example is that, in the second term from the right, the $3$ forces its run decomposition term to have a lower index in the weak descent composition. We avoid this pitfall by instead taking the product
\begin{displaymath}
  \key_{(0,2,1,0)} \key_{(0,0,1,1)} = \key_{(0,3,1,1)} + \key_{(1,2,1,1)} + \key_{(0,3,2,0)} + \key_{(0,2,2,1)},
\end{displaymath}
and the expansion is positive, as desired. This leads to a general positivity result.

\begin{theorem}
  For $\lambda$ a partition of length at most $n$, let $a = a(\lambda,n)$ be the weakly increasing weak composition of length $n$ that sorts to $\lambda$. For $b$ any weak composition of length $n$, the maps $\{d_i\}$ give a weak dual equivalence for $(\Key(b \otimes a_{\lambda}),\des)$. In particular,
  \begin{equation}
    \key_{b} s_{\lambda}(x_1,\ldots,x_n) = \key_b \key_{a_{\lambda}} = \sum_{d} c_{b,a}^{d} \key_{d},
    \label{e:LRR-key}
  \end{equation}
  where $c_{b,a}^{d}$ is the number of yamanouchi key tableaux of shape $b\otimes a$ that rectify to $d$.
  \label{thm:LRR-key}
\end{theorem}

\begin{proof}
  By Proposition~\ref{prop:LRR-key-limit}, the maps $\{d_i\}$ are involutions, and $d_i$ and $d_j$ commute for $|i-j|\geq 3$. Therefore we need only consider weak descent compositions of restricted dual equivalence classes, and for this we must consider how the weak descent composition can differ from the case of a single key tableau. Note that the key tableaux conditions ensure that $\Key_{a_{\lambda}}$ has decreasing rows (left to right) and columns (top to bottom). The only place for discrepancy with Theorem~\ref{thm:deg-key} is if some entry, say $i$, in $a_{\lambda}$ lies in the same run block but in a lower row than an entry, say $j>i$, in $b$. However, in this case, there is necessarily an entry, say $k$, immediately above $i$ and $k>j$ since $k>i$ and not in the same run block. Therefore, in constructing the weak descent composition, the run block containing $k$ will be indexed by the row above $i$, or lower, and so the run block for $j$ will be indexed by the row of $i$, or lower. Therefore we may equivalently define the weak descent compositions by taking the largest entry, instead of the lowest entry, and now the definition coincides with that for a single key tableaux. Therefore the result follows from Theorem~\ref{thm:deg-key}. 
\end{proof}

Notice that the proof of Theorem~\ref{thm:LRR-key} is dependent on placing the partition key diagram to the right of the arbitrary key diagram. While the generating polynomials clearly commute, the obvious bijection between $\Key(a_{\lambda}\otimes b)$ and $\Key(b \otimes a_{\lambda})$ that swaps the two tableaux is not $\des$-preserving.

Since Schubert polynomials are polynomial representatives for Schubert classes in the cohomology ring, the structure constants for Schubert polynomials, $c_{u,v}^{w}$, defined by
\begin{displaymath}
  \schubert_{u} \cdot \schubert_{v} = \sum_{w} c_{u,v}^{w} \schubert_{w},
\end{displaymath}
enumerate flags in a suitable triple intersection of Schubert varieties. Therefore these so-called \emph{Littlewood--Richardson coefficients} are known to be nonnegative. A fundamental open problem in Schubert calculus is to find a \emph{positive} combinatorial construction for $c_{u,v}^{w}$.

Recalling that in the special case of the Grassmannian subvariety, Schubert polynomials are Schur polynomials \cite{LS82} which, in turn, are key polynomials, we have the following geometrically significant corollary to Theorem~\ref{thm:LRR-key}.

\begin{corollary}
  Given partitions $\mu,\nu$ and positive integers $m,n$, let $u=v(\mu,m), v=v(\nu,n)$ be the corresponding grassmannian permutations. Then we have
  \begin{equation}
    \schubert_{u} \schubert_{v} = \sum_{d} c_{a,b}^{d} \key_d,
  \end{equation}
  where $a=a(\mu,m), b=a(\nu,n)$ and $c_{a,b}^{d}$ is the number of Yamanochi key tableaux of shape $a\otimes b$ that rectify to $d$.
  \label{cor:grass-grass}
\end{corollary}

That is, we have a combinatorial rule for the key expansion of any arbitrary product of grassmannian Schubert polynomials. Note that the terms on the right side are not, in general, Schur polynomials. For example,
\begin{displaymath}
  s_{(1,1)}(x_1,x_2,x_3) s_{(1)}(x_1,x_2) = \key_{(1,1,1)} + \key_{(0,2,1)},
\end{displaymath}
and the latter term on the right is not symmetric.

%
%

\bibliographystyle{amsalpha} 
\bibliography{weak_dual.bib}

\end{document}